\documentclass[a4paper,11pt]{article}

\usepackage{amssymb}
\usepackage{amsthm}
\usepackage{tikz}
\usepackage{tikz,pgfplots}
\usepackage{tkz-graph}
\usepackage{tkz-berge}
\usetikzlibrary{decorations.markings}
\usepackage{amsmath,amssymb}
\usepackage{todonotes}
\usepackage[T1]{fontenc}
\pgfplotsset{compat=newest}
\usepackage{cite}
\usepackage{color}
\usepackage{hyperref}

\hypersetup{colorlinks=true}

\hypersetup{colorlinks=true, linkcolor=blue, citecolor=blue,urlcolor=blue}

\usepackage{graphicx}
\usepackage{placeins,caption}

\usepackage[top=3cm,bottom=3cm,left=2.8cm,right=2.8cm]{geometry}
	\newtheorem{theorem}{Theorem}[section]
	\newtheorem{lemma}[theorem]{Lemma}
	\newtheorem{corollary}[theorem]{Corollary}
	\newtheorem{proposition}[theorem]{Proposition}

	\theoremstyle{definition}

        \newtheorem{claim-proof}{Claim}
	\newtheorem{definition}[theorem]{Definition}
	\newtheorem{remark}[theorem]{Remark}

        \newcommand{\modx}[1]{\bmod \, #1}

\begin{document}
	
\title{Strong Edge-Based Roman Domination in Fuzzy Graphs} 

\author{J. C. Valenzuela-Tripodoro$^{a}$\thanks{Email: \texttt{jcarlos.valenzuela@uca.es} ORCID:0000-0002-6830-492X {\bf Corresponding Author}}
 \and
P. Garc\'ia-V\'azquez$^{b}$\thanks{Email: \texttt{pgvazquez@us.es} ORCID:0000-0001-9208-9409}
%\and \makebox[15cm]{ \ \ } 
\and
M. Cera L\'opez$^{b}$ \newline \thanks{Email: \texttt{mcera@us.es} ORCID: 0000-0002-3343-736X}}

\maketitle

\begin{center}
$^a$ Departamento de Matem\'aticas, Universidad de C\'adiz, Algeciras Campus, Spain \\
	\medskip
$^b$ Departamento de Matem\'atica Aplicada I, Universidad de Sevilla, Spain
    \\
	\medskip
\end{center}
		
\begin{abstract}

This work is related to the extension of the well-known problem of Roman domination in graph 
theory to fuzzy graphs. A variety of approaches have been used to explore the concept of 
domination in fuzzy graphs. This study uses the concept of strong domination, considering 
the weights of the strong edges. We introduce the strong-neighbors Roman domination number 
of a fuzzy graph and establish some correlations with the Roman domination in graphs. The strong-neighbors 
Roman domination number is determined for specific fuzzy graphs, including complete and complete
bipartite fuzzy graphs. Besides, several general bounds are given. In addition, 
we characterize the fuzzy graphs that reach the extreme values with particular attention to fuzzy strong 
cycles and paths. 
\end{abstract}


\noindent
AMS Subj.\ Class.\ (2020): 05C12, 05C76

\section{Introduction}
Zadeh introduced fuzzy set theory in 1965~\cite{Z65}, paving the way for fuzzy logic. Fuzzy logic tackles the 
vagueness and uncertainty inherent in human cognition, especially in areas like pattern recognition, information 
communication, and abstraction. It suggests assigning a membership grade between 0 and 1 to elements within 
subsets of a universal set. This approach of ambiguity defined by a membership assigned to each element of 
a set through a fuzzy value has found application in various fields including computer science, AI, decision-making, 
information theory, systems engineering, control systems, expert systems, pattern analysis, business management, 
operational research, and robotics (see~\cite{Z01}). There are also advances of fuzzy set theory in mathematics fields as topology, 
abstract algebra, geometry, analysis, and especially graph theory~\cite{MN01}. The work by Mordeson and 
Nair~\cite{MN00} provides a comprehensive overview of the applications of fuzzy graph theory in areas such 
as cluster analysis, database theory, network analysis, and information theory.

Graph theory is essential in analysing and designing networks, meeting 
specific criteria and demands. The topology of a network can be modeled as 
an undirected graph $G=(V,E)$, where $V$ represents the set of vertices and 
$E$ represents the set of all connecting links between pairs of vertices. The 
communication efficiency and the fault tolerance of a network have been
 intensively studied in graph theory, because graph parameters can
 theoretically model many real measures in a network. To determine the
 strength of connectedness of a graph, numerous connectivity parameters
 have been defined and studied by considering not only the number of
 vertices or links that have to fail but also how the resultant graph by the
 disruption is (see, for instance, \cite{GYZ13}). In a system operating in real-time, 
it is imperative to restrict information delays, as any message received
 beyond the set limit may become obsolete. This prompts the query of
 determining or approximating the number of routes that ensure effective
 transmission of information within the required timeframe. These are only
 some examples of real problems which are treated through graphs by
 measuring distances, computing disjoint paths or minimizing vertex or edge
 disruptions. There are important contributions to the behavior of networks
 modeled by undirected and non-weighted graphs, but the results are more
 limited when the vertices and/or the links have different relative importance
 and therefore, they need to be degraded by assigning membership values. In
 such cases, fuzzy graph theory plays a decisive role.

Rosenfeld~\cite{R75} pioneered the development of fuzzy equivalents for fundamental 
graph-theoretic concepts like bridges, paths, cycles, trees, and connectivity, elucidating
 various associated properties.  Since then, a lot of work has continued studying graph
 parameters in fuzzy graphs. So, we find contributions, among others, 
in connectivity~\cite{AMMR18, JMM22, MS10, MS13, PMM24} and average edge
 connectivity~\cite{MLL23}, coloring~\cite{SPP16, MGP20}, distance~\cite{MM16}, and
 other invariants~\cite{GERHK23, IP23}

Our contribution is related to the well-known problem of Roman domination in graph 
theory, applied to fuzzy graphs. The concept of domination in fuzzy graphs has been 
explored through various approaches. Somasundaram and Somasundaram~\cite{SS98} 
studied domination and total domination in fuzzy graphs using only effective edges. Nagoor 
Gani and Chandrasekaran~\cite{NC06} introduced domination in fuzzy graphs, as the 
number of vertices in a dominating set, by using strong edges. And Manjusha and Sunitha~\cite{MS15}
 defined the domination number of fuzzy graphs by means of the weight of strong edges.  
Later, they used it to study the strong node covering number of a fuzzy graph~\cite{MST19}. 

We will utilize the domination concept proposed by Manjusha and Sunitha to introduce the 
Roman domination number for a fuzzy graph, incorporating the weights of strong edges. 
The introduction  to the notion of Roman domination in graphs traces back to Cockayne et 
al.~\cite{CDHH04}, who drew inspiration from a historical defensive strategy attributed to 
the era of Emperor Constantine I The Great. (refer to ~\cite{S99}). In this strategy, during 
sudden attacks, each vulnerable location within the Roman Empire was required to have a 
fortified neighbor, capable of deploying two legions, ensuring that the stronger city could 
send reinforcements to defend the besieged location without compromising its own 
defense.

This same approach can be useful in problems such as searching for the best place to locate 
a public service or a business, but, until now, Roman domination has only been applied, with 
its variants, to graphs with unweighted vertices and edges. In real networks, not all points or 
all connections are equally important, which suggests working with weighted graphs, where 
a membership value is assigned to each vertex and each edge according to its importance. 

In this work, we introduce the {\it strong-neighbors Roman domination number} of a fuzzy 
graph and prove some relationships with other well-known
domination parameters. We prove some bounds, we determine the value of the strong-
neighbors Roman domination number for specific fuzzy graphs and 
also we characterize the fuzzy graphs for which the extreme values are attained. 

\section{Definitions and terminologies}

In graph theory, a graph G is defined as a pairing $(V, E)$, where $V$ denotes a collection of 
vertices and $E$ represents a set of unsorted vertex pairs, known as edges. Introducing a 
novel dimension, a fuzzy graph $G = (V,\sigma,\mu)$ emerges as a triplet, comprising a non-
empty $V$ set along with a duo of membership functions $\sigma:V\longrightarrow [0, 1]$ 
and $\mu: V\times V \longrightarrow [0, 1]$. In this construct, $\mu(x, y)$ represents, at 
most, the minimal value among the membership values attributed to its extreme vertices, 
ensuring that  $\mu(x,y) \le  \sigma(x) \wedge \sigma(y)$ for all $x,y \in V$. It is important 
to note that $V$ is conventionally treated as a finite set unless stated otherwise. Within this 
framework, $\sigma$ is denoted as the fuzzy vertex set of $G$, while $\mu$ represents the 
fuzzy edge set. It is evident that $\mu$ acts as a fuzzy relation on $\sigma$. Consequently, 
for simplicity in notation, $G$ or $(\sigma, \mu)$ is employed to denote the fuzzy graph 
$G : (V, \sigma,\mu)$. Additionally, we denote the underlying graph by $G^*:(\sigma^*,
\mu^*)$, where $\sigma^*=\{u\in V:\sigma(u)>0\}$ and $\mu^*=\{(u,v)\in V\times V:
\mu(u,v)>0\}$.

Consider two vertices, $u$ and $v$, within a fuzzy graph $G : (V,\sigma,\mu)$. Should 
$\mu(u,v)$ exceed zero, we refer to $u$ and $v$ as {\it adjacent} or {\it neighbors}. Within 
$G$, a path $P$ linking $u$ and $v$, known as a $u-v$ path, consists of a sequence of 
vertices $u=x_0, x_1, \ldots , x_n=v$ where $\mu(x_{i-1}x_i)>0$, for $i = 1, \ldots , n$, 
defining $n$ as the length of the path. These consecutive pairs represent the edges of the 
path. The {\it strength} of $P$, denoted as $s(P)$, is defined as $\wedge{\mu(x_{i-1},x_i):i=1,\ldots,n}$, 
signifying the weight of the weakest edge. Furthermore, the {\it strength of connectedness} 
between vertices $u$ and $v$ within $G$ is established as the maximum strength among all 
$u-v$ paths, denoted as ${CONN}_G(u,v)$. A {\it strongest $u-v$ path} is a path with 
strength ${CONN}_G(u,v)$. A fuzzy graph $G$ is connected if for every $x$, $y$ in 
$\sigma^*$, ${CONN}_G(u,v)>0$. If $u$ and $v$ are neighbors, we call the edge $(u,v)$ {\it 
strong} if $\mu(u,v)= {CONN}_G(u,v)$. In such a case, we say that $u$ and $v$ are strong 
neighbors. The set of all neighbors of a vertex $u$ is denoted by $N(u)$, whereas the set of 
strong neighbors of $u$ is denoted by $N_s(u)$ and it is called the {\it strong neighborhood 
of $u$}. The {\it closed strong neighborhood of $u$} is $N_s[u]=N_s(u)\cup \{u\}$. The {\it 
degree} and the {\it strong degree} of a vertex $u$ are, respectively, $d(u)=\sum_{v\in 
N(u)}\mu(u,v)$ and $d_s(u)=\sum_{v\in N_s(u)}\mu(u,v)$. That is, the degree of $u$ is the 
sum of the membership values of all edges incident at $u$, whereas the strong degree is 
defined by computing only the membership values of all strong edges incident at $u$. The 
{\it minimum strong degree} and the {\it maximum strong degree} of $G$ are $\delta_s(G)=
\wedge\{d_s(u):u\in V\}$ and $\Delta_s(G)=\vee\{d_s(u):u\in V\}$, respectively. The {\it 
strong neighborhood degree} of a vertex $u$ is defined as $d_{SN}(u)=\sum_{v\in 
N_s(u)}\sigma(v)$. The {\it minimum strong neighborhood degree} of $G$ is $\delta_{SN}
(G)=\wedge\{d_{SN}(u):u\in V\}$ and the {\it maximum strong neighborhood degree} of $G
$ is $\Delta_{SN}(G)=\vee\{d_{SN}(u):u\in V\}$. 
Let us denote by $\mu_s(u)=\wedge \{\mu(u,v):v\in N_s(u)\}$.
An edge $uv$ is called {\it effective} or {\it $M$-strong} if $\mu(u,v)=\sigma(u)\wedge 
\sigma(v)$. In such a case, $u$ and $v$ are effective neighbors. A fuzzy graph $G : (V,
\sigma,\mu)$ with  $|\sigma^*|=n$ is called a {\it complete fuzzy graph}, and it is denoted 
by $K_n$, if $\mu(u,v)=\sigma(u)\wedge \sigma(v)$, for every $u,v\in V$, that is, if every 
two vertices of $G$ are effective neighbors. A fuzzy graph $G = (V,\sigma,\mu)$ is said to be 
{\it bipartite} if $V$ can be partitioned into two nonempty subsets,  $X$ and $Y$, so that 
$\mu(u,v)=0$ if either  $u,v\in X$ or $u,v\in Y$. Further, if $\mu(u,v)=\sigma(u)\wedge 
\sigma(v)$, for every  $u\in X$ and $v\in Y$, then $G$ is called a {\it complete bipartite 
graph} and is usually denoted by $K_{\sigma_1,\sigma_2}$, where  $\sigma_1$ and 
$\sigma_2$ are the restrictions of $\sigma$ to $X$ and $Y$, respectively. When $\mu(u,v)=
0$ for all $v\in V\setminus\{u\}$, we say that $u$ is an {\it isolated vertex}.

The \emph{order} and \emph{size} of a fuzzy graph $G : (V,\sigma,\mu)$ are determined by 
$p=\sum_{x\in V}\sigma(x)$ and $q=\sum_{(x,y)\in V\times V}\mu(xy)$, correspondingly. 
For any subset $S\subseteq V$, the number of elements of $S$ is denoted by $|S|$, and the 
scalar quantity $|S|_s$ is calculated as $|S|_s=\sum_{u\in S}\sigma(u)$. The 
\emph{complement of $G$} is $\overline{G}=(V,\sigma,\overline{\mu})$, where 
$\overline{\mu}(x,y)=\sigma(x)\wedge \sigma(y)-\mu(x,y)$ for all $x,y\in V$. It is worth 
noting that the complement of a complete fuzzy graph results in a fuzzy graph wherein each 
vertex stands independently.

We denote by $\mu_{\min}=\mu_{\min}(G)=\min\{\mu(u,v):v\in N_s(u),u\in V\}$ and 
$\mu_{\max}=\mu_{\max}(G)=\max\{\mu(u,v):v\in N_s(u),u\in V\}$ the minimum and the 
maximum weight of a strong edge in $G$, respectively, and by $\overline{\mu}_{\min}=
\mu_{\min}(\overline{G})$ and $\overline{\mu}_{\max}=\mu_{\max}(\overline{G})$.

In graph theory, given a graph $G:(V,E)$, every vertex $v\in V$ dominates itself and each of 
its neighbors. So, a subset $S\subseteq V$ is a dominating set of $G$ if every vertex of $V-S
$ is neighbor of some vertex of $S$. The minimum cardinality of a dominating set of $G$ is 
called the {\it domination number of $G$} and it is denoted by $\gamma(G)$. One of the 
most important variants of domination in graphs is the {\it Roman domination},  introduced 
by Cockayne et al.~\cite{CDHH04}. In a graph $G : (V,E)$, a {\it Roman dominating function} 
is a mapping $f: V \longrightarrow \{0,1,2\}$ that satisfies the condition where each vertex 
$u$ with $f(u) = 0$ must be adjacent to at least one vertex $v$ such that $f(v) = 2$. The 
total weight of a Roman dominating function is denoted by $w(f) = \sum_{u\in V} f(u)$. The 
smallest weight among all Roman dominating functions on a graph $G$ is termed the {\it 
Roman domination number} of $G$, $\gamma_R(G)$.

In fuzzy graph theory, domination in graphs was firstly extended to fuzzy graphs by Nagoor 
Gani and Chandrasekaran~\cite{NC06} using strong edges, as follows. Given a fuzzy graph 
$G:(V,\sigma,\mu)$, every vertex $v\in V$ strongly dominates itself and each of its strong 
neighbors. So, a subset $D\subseteq V$ is a strong dominating set of $G$ if every vertex of 
$V-D$ is a strong neighbor of some vertex of $D$. Nagoor Gani and Vijayalakshmi~
\cite{NV11} defined the strong dominating number of a fuzzy graph $G$ as the scalar 
cardinality of a strong dominating set of $G$ with minimum number of vertices. As 
Manjusha and Sunitha said in~\cite{MS15}, in fuzzy location problems in networks, where 
strong domination can be applied, the membership values of strong edges in fuzzy graphs 
give more optimum results for strong domination number than using membership values of 
vertices. They proposed in~\cite{MS15} an analogous definition to that of Nagoor Gani and 
Vijayalakshmi~\cite{NV11} considering the minimum weight of all strong edges incident in 
each vertex of the dominating set.

\begin{definition}\label{dom}
    (See~\cite{MS15}) The weight of a strong dominating set $D$ is defined as $W(D)=\sum_{u\in D}\mu_s(u)$. 
    The strong dominating number of a fuzzy graph $G$ is defined as the minimum weight among all the strong
    dominating sets of $G$ and it is denoted by $\gamma_s(G)$. A minimum strong dominating set in a fuzzy 
    graph $G$ is a strong dominating set having minimum weight.
\end{definition}

 We introduce a Roman-type analogous concept to the one given by Definition~\ref{dom}, called 
 {\it strong-neighbors Roman domination number} of a fuzzy graph, by considering the criterion 
 of Manjusha and Sunitha~\cite{MS15}.  

\begin{definition}\label{Romdom}
    Let $G:(V,\sigma,\mu)$ a fuzzy graph. A strong-neighbors Roman dominating function (SNRDF) on $G$ 
    is defined as a function $f:V\rightarrow \{0,1,2\}$ satisfying the condition that every vertex $u\in V$ 
    for which $f(u) = 0$ has at least one strong neighbor $v\in N_s(u)$ for which $f(v) = 2$. The weight 
    of a SNRDF is the value $w(f)=\sum_{u\in V}f(u)\mu_s(u)$. The strong-neighbors Roman domination number 
    of $G$, denoted by $\gamma_{snR}(G)$, is the minimum weight among all possible SNRDFs on $G$.
\end{definition}

Associated to a SNRDF $f$ on a fuzzy graph $G:(V,\sigma,\mu)$ we can construct a partition of $V$ into 
the sets $(V_0,V_1,V_2)$, where $V_i=\{u\in V:f(u)=i\}$, for $i=0,1,2$. There is a bijective correspondence 
between $f$ and its associated partition. Thus, in what follows, whenever we consider an SNRDF $f$, 
we will refer to its associated partition $(V_0,V_1,V_2)$. An SNRDF on $G$ with weight $\gamma_{snR}$ 
will be called a $\gamma_{snR}$-function on $G$ or a $\gamma_{snR}(G)$-function. 

\section{General results on strong-neighbors Roman domination number}
We first find a relationship between the strong-neighbors Roman and the strong domination numbers.

\begin{theorem}\label{romcotgen}
    For every fuzzy graph $G:(V,\sigma,\mu)$, it follows that $\gamma_s(G)\le \gamma_{snR}(G)\le 2\gamma_s(G)$.
\end{theorem}
\begin{proof}
    First, let us see that $\gamma_s(G)\le \gamma_{snR}(G)$. Let $f$ be a $\gamma_{snR}$-function and 
    consider the associated partition on $V$, $(V_0, V_1, V_2)$. As every vertex of $V_0$ has at least 
    one strong neighbor in $V_2$, we can assure that $D=V_1\cup V_2$ is a strong dominating set on $G$ and therefore, 
    $$
        \begin{array}{rcl}
            \displaystyle \gamma_s(G) \le W(D) & = &  
            \displaystyle \sum_{u\in V_1}\mu_s(u)+\sum_{u\in V_2}\mu_s(u) 
                \le  \displaystyle \sum_{u\in V_1}\mu_s(u)+2\sum_{u\in V_2}\mu_s(u) \\[2ex] 
                & = & \displaystyle \sum_{i=0}^2 \sum_{u\in V_i} f(u) \mu_s(u)
                 =  w(f)=\gamma_{snR}(G).
        \end{array}$$
    
    Second, we prove that $\gamma_{snR}(G)\le 2\gamma_s(G)$. Let $D$ be a $\gamma_s(G)$-set of $G$ 
    and consider the function $f:V\rightarrow \{0,1,2\}$ defined as $f(u)=2$, if $u\in D$ and $f(u)=0$, 
    otherwise. Clearly, $f$ is a SNRDF on $G$, yielding that 
    $\displaystyle \gamma_{snR}(G)\le w(f)=\sum_{u\in D}2\mu_s(u)=2W(D)=2\gamma_s(G)$.
\end{proof}

Now, we will approach the families of fuzzy graphs for which $\gamma_s$ and $\gamma_{snR}$ reach the 
same value. Before that, the following remark is necessary.

\begin{remark}\label{Garifue}
 The following assertions hold,
    \begin{enumerate}
        \item [(a)] If $G$ is a fuzzy graph with size $q>0$ then $G$ has some strong edge.
        \item[(b)] For a $\gamma_{snR}$-function such that $V_1$ has the minimum possible cardinality, 
        there are no strong neighbors in $V_1.$ 
    \end{enumerate}
\end{remark}
\begin{proof}
    (a) Observe that $G$ has some edge, because $q>0$. Let $(u,v)$ be the edge of $G$ with the 
    largest weight. If $(u,v)$ is not strong, then there exists a strongest $u-v$ path $P$ 
    with strength $s(P)>\mu(u,v)$. This means that the weight of every edge of $P$ is greater 
    than $\mu(u,v)$, which contradicts the fact that $(u,v)$ is the edge of $G$ with the largest 
    weight.
    
    (b) Let $f$ be a $\gamma_{snR}$-function such that $|V_1|$ is minimum. 
    We reason by contradiction supposing that $(u,v)$ is an edge joining two strong neighbors of $V_1$. 
    Without loss of generality, we may assume that $\mu_s(u) \le \mu_s(v)$. Then we may define a 
    new function $g$ as follows: $g(u)=2, g(v)=0$ and $g(z)=f(z)$ for all $z\in V\setminus\{u,v\}$. 
    The vertex $u$ is dominated by $v$ and so $g$ is a strong-neighbors Roman domination function 
    with $w(g)=w(f)+\mu_s(u)-\mu_s(v)\le w(f)$ but having less number of vertices with a 1 label, 
    which is a contradiction.
\end{proof}

A strong dominating set $D$ on a fuzzy graph $G:(V,\sigma,\mu)$ is called a {\it minimal strong 
dominating set} if no proper subset of $D$ is a strong dominating set on $G$. The following theorem, 
proved in~\cite{NC06}, gives two properties that characterize a strong dominating set on a fuzzy 
graph as minimal.

\begin{theorem}\label{dommal}
(See~\cite{NC06}) Let $G:(V,\sigma,\mu)$ be a fuzzy graph. Then a dominating set $D$ on $G$ is 
minimal if and only if every $u\in D$ satisfies at least one of these two conditions,
\begin{itemize}
\item[1.] $u$ is not a strong neighbor of any vertex in $D$.
\item[2.] There is a vertex $v\in V-D$ such that $N_s(v)\cap D=\{u\}$.
\end{itemize}
\end{theorem}

First, we characterize the fuzzy graphs for which the lower bound for $\gamma_{snR}$ given 
in Theorem~\ref{romcotgen} is attained.

\begin{theorem}
    Let $G:(V,\sigma,\mu)$ be a fuzzy graph with $n$ vertices. Then $\gamma_s(G)=\gamma_{snR}(G)$ 
    if and only if $G=\overline{K_n}$.
\end{theorem}
\begin{proof}
    If $G$ has no strong edges then the only strong dominating set $D$ is $V$ and $W(V)=0$. 
    Besides every SNRD function $f$ on $G$ has weight $w(f)=0$. Thus, $\gamma_s(G)=0=\gamma_{snR}(G)$. 

    Assume that $\gamma_s(G)=\gamma_{snR}(G)$. According to Remark~\ref{Garifue}, we only must 
    show that $G$ has no strong edges. Let $f$ be a $\gamma_{snR}$-function on $G$ and consider 
    the associated partition $(V_0,V_1,V_2)$ of $V$. As every vertex $u$ with $f(u)=0$ must have 
    a strong neighbor $v$ with $f(v)=2$, we can assure that $V_1\cup V_2$ is a strong dominating 
    set of $G$. Then 
    $$
    \begin{array}{rcl}
        \displaystyle \gamma_{snR}(G)=\gamma_s(G) & \le & W(V_1\cup V_2) 
                         =  \displaystyle \sum_{u\in V_1}\mu_s(u)+\sum_{u\in V_2} \mu_s(u)  \\[2ex] 
                        & \le & \displaystyle \sum_{u\in V_1}\mu_s(u)+2\sum_{u\in V_2}\mu_s(u)
                         =  \displaystyle \sum_{i=0}^2 \sum_{u\in V_i}f(u)\mu_s(u) 
                         =  w(f)= \gamma_{snR}(G),
    \end{array}
    $$ 
    following that 
    \begin{equation}\label{V2ais} 
        \displaystyle \sum_{u\in V_2} \mu_s(u)=0.
    \end{equation} 
    
    This means that $N_s(u)=\emptyset$, for every vertex $u\in V_2$ and therefore, $V_0=\emptyset$. 
    As $\gamma_s(G)=\gamma_{snR}(G)$, from~(\ref{V2ais}) it follows that $
    \displaystyle \gamma_{snR}(G)=\sum_{u\in V_1}\mu_s(u)=\gamma_s(G).$ 
    
By item (b) of Remark~\ref{Garifue}, there are no strong neighbors in $V_1$. Thus,  $G=\overline{K_n}$.    
\end{proof}

The next theorem also provides an upper and a lower bounds for $\gamma_{snR}$ in terms of the values of the fuzzy 
vertex set $\sigma$ and the minimum membership value of the strong edges.

\begin{theorem}\label{romcot}
    Let $G:(V,\sigma,\mu)$ be a fuzzy graph of order $p$. Then 
    $$\displaystyle 2\cdot \wedge\{\mu(u,v):u,v\in V\} \le 
        \gamma_{snR}(G)\le p-\max_{v\in V}\left\{d_{SN}(v)-\sigma(v)\right\}.$$ 
\end{theorem}
\begin{proof}
    The first inequality holds if $\wedge\{\mu(u,v):u,v\in V\}=0$. So, assume that $q>0$ and let $x,y\in V$ 
    be such that $0<\mu(x,y)=\wedge\{\mu(u,v):u,v\in V\}$. Consider any $\gamma_{snR}$-function $f$ with partition 
    $(V_0,V_1,V_2)$ on $V$. Since by Remark~\ref{Garifue} there are strong neighbors in $G$, we deduce that 
    $V_1\cup V_2\ne \emptyset$ and therefore, $\gamma_{snR}(G)=w(f)\ge 2\cdot \wedge \{\mu(u,v):u,v\in V\}$.   
    
    To prove the another inequality, for any vertex $v\in V$ let us consider the function $f_v:V\rightarrow \{0,1,2\}$ 
    defined as $f_v(v)=2$, $f_v(u)=0$ if $u\in N_s(v)$, and $f_v(u)=1$ otherwise. Clearly $f_v$ is a SNRDF on $G$ and 
    $$
    \begin{array}{rcccl} 
        w(f_v) & = & \displaystyle 2 \mu_s(v)+\sum_{u\not\in N_s[v]}\mu_s(u)   
               & \le & \displaystyle 2\sigma(v)+\sum_{u\not\in N_s[v]}\sigma(u)\\[2.75ex] 
               & = & \displaystyle \sigma(v)+p-\sum_{u\in N_s(v)}\sigma(u)       
               & = & p-\left(d_{SN}(v)-\sigma(v)\right).
    \end{array}
    $$ 
    Hence, $\displaystyle \gamma_{snR}\le \min_{v\in V}w(f_v) \le 
    \min_{v\in V} \{p-\left(d_{SN}(v)-\sigma(v)\right)\} =p-\max_{v\in V}\left\{d_{SN}(v)-\sigma(v)\right\}$.
\end{proof}

The lower bound is tight, for example, for a complete fuzzy graph (see Corollary~\ref{complete}). 
Besides, it is not difficult to check that the upper bound is also tight, see the graph depicted in 
Figure~\ref{fig:uppstar}.

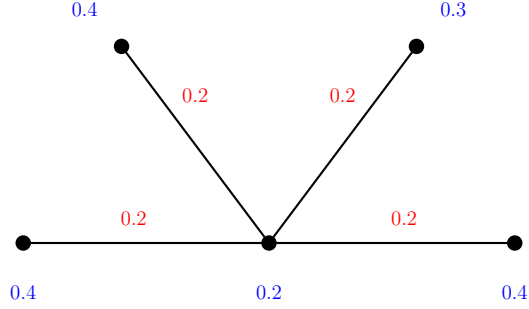
\begin{figure}[ht]
\centering
\begin{tikzpicture}[scale=.65, transform shape, style=thick]
% Node styles
\tikzstyle{vert_black}=[fill=black, draw=black, shape=circle, scale=0.75]
\tikzstyle{none}=[fill=none, draw=none, shape=circle]

            \node [style={vert_black}] (0) at (0, 0) {};
		\node [style={vert_black}] (1) at (-5, 0) {};
		\node [style={vert_black}] (2) at (5, 0) {};
		\node [style={vert_black}] (3) at (3, 4) {};
		\node [style={vert_black}] (4) at (-3, 4) {};
		\node [style=none] (5) at (0, -1) {\large \color{blue} $0.2$};
		\node [style=none] (6) at (5, -1) {\large \color{blue} $0.4$};
		\node [style=none] (7) at (3.75, 4.75) {\large \color{blue} $0.3$};
		\node [style=none] (8) at (-3.75, 4.75) {\large \color{blue} $0.4$};
		\node [style=none] (9) at (-5, -1) {\large \color{blue} $0.4$};
		\node [style=none] (10) at (-1.5, 3) {\large \color{red} $0.2$};
		\node [style=none] (11) at (1.5, 3) {\large \color{red} $0.2$};
		\node [style=none] (12) at (2.75, 0.5) {\large \color{red} $0.2$};
		\node [style=none] (13) at (-2.75, 0.5) {\large \color{red} $0.2$};
  
		\draw (1) to (0);
		\draw (0) to (2);
		\draw (3) to (0);
		\draw (4) to (0);
  
\end{tikzpicture}
    \caption{A graph for which the upper bound is reached.}\label{fig:uppstar}
\end{figure}

The following proposition establishes some particular properties of the partition 
associated to a $\gamma_{snR}$-function. 

\begin{proposition}\label{propiedades}
     Let $G:(V,\sigma,\mu)$ be a fuzzy graph and let $f$ be a $\gamma_{snR}$-function on $G$ with 
     associated partition $(V_0,V_1,V_2)$. Then the following assertions hold,
     \begin{itemize}
         \item[(a)] If $u,v\in V_1$ are strong neighbors, then $\displaystyle \mu_s(u)=\mu_s(v)$.
         \item[(b)] No strong edge of $G$ joins $V_1$ and $V_2$.
         \item[(c)] $V_2$ is a $\gamma_s$-set of $G[V_0\cup V_2]$
     \end{itemize}
\end{proposition}
\begin{proof}
    (a) We reason by contradiction assuming that $\displaystyle \mu_s(u)<\mu_s(v)$ and consider the function 
    $g:V\rightarrow \{0,1,2\}$ defined as $g(u)=2$, $g(v)=0$, and $g(z)=f(z)$ for every $z\in V\setminus\{u,v\}$. 
    Clearly, $g$ is a SNRDF on $G$ and its weight is 
    $w(g) = w(f)-\left(\mu_s(u)+\mu_s(v)\right)+2 \mu_s(u)= w(f)+\mu_s(u)-\mu_s(v)<w(f)$, which is a contradiction 
    to the fact that $f$ is a $\gamma_{snR}$-function.   

    (b) Suppose that a strong edge joins $u\in V_1$ and $v\in V_2$. Then by reassigning $f(u) = 0$ and keeping 
    all other values of $f$ to be the same, we find a new SNRDF on $G$ with a smaller weight than the weight 
    of $f$, a contradiction.

    (c) As each vertex $u\in V_0$ has at least one strong neighbor $v\in V_2$, it is clear that $V_2$ is a strong 
    dominating set of the subgraph induced by $V_0\cup V_2$, say $G[V_0\cup V_2]$. If $V_2$ is not a $\gamma_s$-set 
    of $G[V_0\cup V_2]$ then there exists a dominating set in $G[V_0\cup V_2]$, $V'_2\subseteq V_0\cup V_2,$ such 
    that $W(V'_2)<W(V_2)$. But in this case, the function $g:V\rightarrow \{0,1,2\}$ defined as $g(v)=2$ if 
    $v\in V'_2$, $g(v)=0$ if $v\in (V_0\cup V_2)\setminus V'_2$, and $g(v)=1$ otherwise is a SNRDF on $G$ with 
    partition $\left((V_0\cup V_2)\setminus V'_2,V_1,V'_2\right)$ and its weight is 
    $w(g)=2W(V'_2)+W(V_1)<2W(V_2)+W(V_1)=w(f)$, again a contradiction, because $f$ is a $\gamma_{snR}$-function.  
\end{proof}

Let us denote by $\mathcal{U}_s(G)=\{v\in V: N_s[v]=V\}$ the set of universal vertices in $G$.
\begin{proposition}\label{universal}
    Let $G:(V,\sigma,\mu)$ be a fuzzy graph such that $\mathcal{U}_s(G)\ne\emptyset$. Then 
    $$ \gamma_{snR}(G)=2\min\{\mu_s(v):v\in \mathcal{U}_s(G)\}.$$
\end{proposition}
\begin{proof}
    Let $u\in \mathcal{U}_s(G)$ be such that $\mu_s(u)=\min\{\mu_s(v):v\in \mathcal{U}_s(G)\}$. Let us see that 
    $\mu_s(u)\le \mu_s(v)$, for all $v\in V$. Suppose, on the contrary, that there exists a vertex $w\in V$ such 
    that $\mu_s(w)< \mu_s(u)$. Then there is a vertex $z\in N_s(w)\setminus\{u\}$ such that $\mu(w,z)=\mu_s(w)$. 
    As $u\in \mathcal{U}_s(G)$ and therefore, $N_s[u]=V$, we deduce that $z\in N_s(u)$. But in that case, we find 
    a $w-z$ path $P:w,u,z$ with strengh $s(P)=\wedge\{\mu(w,u),\mu(u,z)\}\ge \mu_s(u)>\mu_s(w)$, yielding that 
    $\mu(w,z)=\mu_s(w)<s(P)\le {CONN}_G(w,z)$. That is, $(w,z)$ is not a strong edge, a contradiction. Hence, 
    $\mu_s(u)\le \mu_s(v)$, for all $v\in V$ and the function $f:V\rightarrow \{0,1,2\}$ defined as $f(u)=2$, 
    and $f(v)=0$ otherwise is a $\gamma_{snR}$-function with weight 
    $\gamma_{snR}(G)=\omega(f)=2\mu_s(u)=2\min\{\mu_s(v):v\in \mathcal{U}_s(G)\}$.  
\end{proof}

\begin{corollary}\label{complete}
    For a complete fuzzy graph $K_n:(V,\sigma,\mu)$, with $|\sigma^*|=n$, we have that
    $\gamma_{snR}(K_n)=2\min\{\mu_s(v):v\in V\}.$
\end{corollary}
\begin{proof}
   By appyling Proposition~\ref{universal}, because all 
   the edges of a complete graph are strong.  
\end{proof}

In general, it is not easy to obtain closed formulae for the exact value of the $\gamma_{snR}$ parameter even when
restricted to well-known graphs. When studying the Roman domination of bipartite graphs $K_{2,r}$ the correct strategy
is always to assign the labels $\{1,2\}$ to the vertices of the class with cardinality two. In case of fuzzy complete
bipartite graphs we cannot assure that this would be the optimal choice (see Figure~\ref{fig:bipart}).

\begin{proposition}
    Let $K_{\sigma_1,\sigma_2}:(X\cup Y, \sigma, \mu)$ be a complete bipartite fuzzy graph with $X=\{x_i\}_{i=1}^{\sigma_1}$ such that 
    $\mu_s(x_i)\le \mu_s(x_{i+1})$ and $Y=\{y_j\}_{j=1}^{\sigma_2}$ such that 
    $\mu_s(y_j)\le \mu_s(y_{j+1})$, with $\sigma_1 \le \sigma_2$. 
    The following assertions hold,

    \begin{itemize}
        \item[(i)] If $\sigma_1=1$ then $\displaystyle \gamma_{snR}(K_{\sigma_1,\sigma_2})=2\mu_s(x_1)$.
        \item[(ii)] If $\sigma_1=\sigma_2=2$ then 
        $$ \gamma_{snR}(K_{\sigma_1,\sigma_2})= 
            \left\{
                \begin{array}{ll} 
                4\mu_s(x_1),  & \mbox{if } 2\mu_s(x_1) \le \min\{\mu_s(x_2),\mu_s(y_2)\} \\[2ex] 
                2\mu_s(x_1)+\min\{\mu_s(x_2),\mu_s(y_2)\}, & \mbox{if } 2\mu_s(x_1) \ge \max\{\mu_s(x_2),\mu_s(y_2)\}.
                \end{array}
            \right.
        $$
        \item[(iii)] If $\sigma_1=2$ and $\sigma_2\ge 3$ then 
                 $\gamma_{snR}(K_{\sigma_1,\sigma_2})= \min\{4\mu_s(x_1),2\mu_s(x_1)+\mu_s(x_2)\}$
                 
        \item[(iv)] If $\sigma_1 \ge 3$ then  $\displaystyle \gamma_{snR}(K_{\sigma_1,\sigma_2})=4 \mu_s(x_1)$.
    \end{itemize}
\end{proposition}

\begin{figure}[ht]
\centering
\begin{tikzpicture}[scale=.75, transform shape, style=thick]
% Node styles
\tikzstyle{vert_black}=[fill=black, draw=black, shape=circle, scale=0.75]
\tikzstyle{none}=[fill=none, draw=none, shape=circle]
\tikzstyle{vert_dest}=[fill=none, draw=red, shape=circle, scale=2]

		\node [style={vert_black}] (5) at (3, 2) {};
		\node [style={vert_black}] (6) at (3, -3) {};
		\node [style={vert_black}] (7) at (7, -0.5) {};
		\node [style={vert_black}] (8) at (7, 4) {};
		\node [style={vert_black}] (9) at (7, -5) {};
		\node [style=none] (12) at (7.75, 4.5) {\large $0.3$};
		\node [style=none] (13) at (7.75, 0) {\large $0.3$};
		\node [style=none] (14) at (7.75, -4.5) {\large $0.3$};
		\node [style=none] (15) at (2, 2.5) {\large $0.1$};
		\node [style=none] (16) at (2, -2.5) {\large $0.3$};
		\node [style={vert_dest}] (21) at (7, 4) {};
		\node [style={vert_dest}] (22) at (3, 2) {};
		\node [style=none] (23) at (2, 1.25) {{\color{red} $f(x_1)=2$}};
        \node [style=none] (24) at (7.5, 3) {{\color{red} $f(y_1)=2$}};
		\node [style={vert_black}] (25) at (-6, 2) {};
		\node [style={vert_black}] (26) at (-6, -3) {};
		\node [style={vert_black}] (27) at (-2, -0.5) {};
		\node [style={vert_black}] (28) at (-2, 4) {};
		\node [style={vert_black}] (29) at (-2, -5) {};
		\node [style=none] (30) at (-1.25, 4.5) {\large $0.3$};
		\node [style=none] (31) at (-1.25, 0) {\large $0.3$};
		\node [style=none] (32) at (-1.25, -4.5) {\large $0.3$};
		\node [style=none] (33) at (-7, 2.5) {\large $0.1$};
		\node [style=none] (34) at (-7, -2.5) {\large $0.1$};
		\node [style={vert_dest}] (35) at (-6, -3) {};
		\node [style={vert_dest}] (36) at (-6, 2) {};
		\node [style=none] (37) at (-7, 1.25) {{\color{red} $f(x_1)=2$}};
		\node [style=none] (38) at (-7, -4) {{\color{red} $f(x_2)=1$}};
		\node [style=none] (39) at (-5, -6) {{\color{red} \large $w(f)=2\cdot 0.1+0.1=0.3$}};
		\node [style=none] (40) at (5, -6) {{\color{red} \large $w(f)=2\cdot 0.1+2\cdot 0.1=0.4$}};
     % New
        \node [style=none] (41) at (-5.25, 2) {{$x_1$}};
        \node [style=none] (42) at (3.75, 2) {{$x_1$}};
        \node [style=none] (43) at (-5.25, -3) {{$x_2$}};
        \node [style=none] (44) at (3.75, -3) {{$x_2$}};
        \node [style=none] (45) at (-2.75, 4) {{$y_1$}};
        \node [style=none] (46) at (-2.75, -0.5) {{$y_2$}};
        \node [style=none] (47) at (-2.75, -5) {{$y_3$}};
        \node [style=none] (48) at (6.25, 4) {{$y_1$}};
        \node [style=none] (49) at (6.25, -0.5) {{$y_2$}};
        \node [style=none] (50) at (6.25, -5) {{$y_3$}};

		\draw (9) to (6);
		\draw (6) to (7);
		\draw (7) to (5);
		\draw (5) to (8);
		\draw (8) to (6);
		\draw (5) to (9);
		\draw (29) to (26);
		\draw (26) to (27);
		\draw (27) to (25);
		\draw (25) to (28);
		\draw (28) to (26);
		\draw (25) to (29);
\end{tikzpicture}
   \vspace{-1.25cm} \caption{Two different strategies depending on the $\mu$ function.}\label{fig:bipart}
\end{figure}
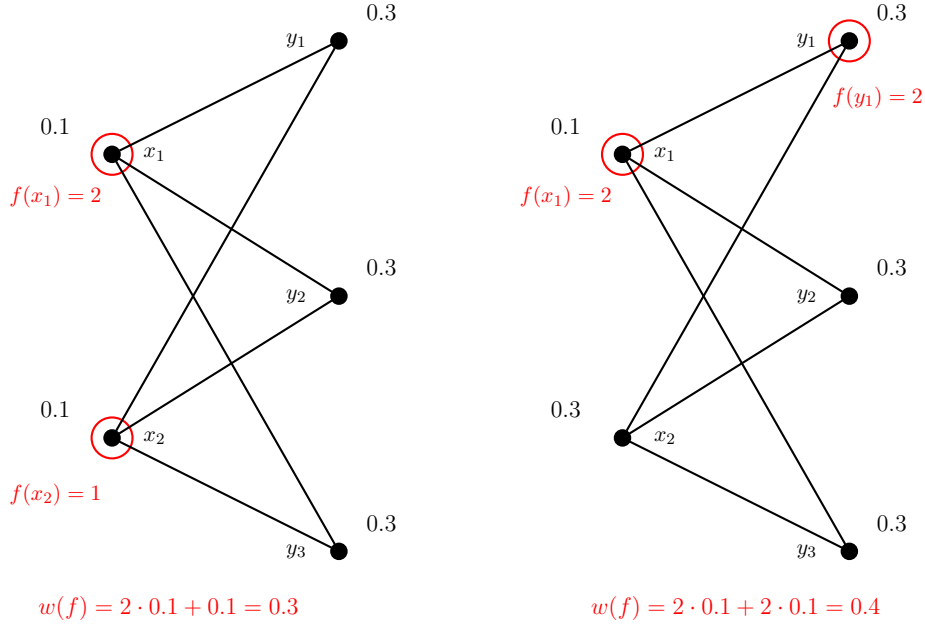

\begin{proof}
    Observe that $\wedge\{\mu(x_i,y_j):i=1,\ldots,\sigma_1,j=1,\ldots,\sigma_2\}=\mu(x_1,y_1)=\mu_s(x_1)=\mu_s(y_1)$.
    
    $(i)$ If $\sigma_1=1$ then $x_1\in \mathcal{U}_s(K_{\sigma_1,\sigma_2})$. By Proposition~\ref{universal}, this item holds.

    $(ii)$ Let $f$ be a $\gamma_{snR}$-function on $K_{\sigma_1,\sigma_2}$ with minimum number of vertices in $V_1$. 
    By Remark~\ref{Garifue}, there are
    no strong neighbors in $V_1$ and hence $|V_1|\le 1.$ As there is no strong edges between $V_1$ and $V_2,$
    there are only two possibilities: either $|V_2|=|V_0|=2$ and $|V_1|=0$ or $|V_2|=|V_1|=1$ and $|V_0|=2.$
    Moreover, if $|V_2|=|V_1|=1$ then both vertices must belong to the same vertex class.
    If $|V_2|=2$ then $w(f)=\sum_{v\in V} f(v) \mu_s(v)=2\mu_s(x_1)+ 2\mu_s(y_1)=4\mu_s(x_1)$ because $f$
    has minimum weight.
    If $|V_2|=|V_1|=1$ then $w(f)=\min\{ 2 \mu_s(x_1)+\mu_s(x_2), 2\mu_s(y_1)+\mu_s(y_2) \}
    =2\mu_s(x_1)+\min\{ \mu_s(x_2) , \mu_s(y_2) \}$ because $f$
    has minimum weight.

    Therefore $\gamma_{snR}(K_{\sigma_1,\sigma_2})=4\mu_s(x_1),$ if $2\mu_s(x_1) \le \min\{ \mu_s(x_2),\mu_s(y_2) \}$ 
    and $\gamma_{snR}(K_{\sigma_1,\sigma_2})=2\mu_s(x_1)+\min\{\mu_s(x_2), \mu_s(y_2)\}$ otherwise.

    $(iii)$ The function $f:V\rightarrow \{0,1,2\}$ defined as $f(x_1)=2$, $f(y_1)=2$ and 
    $f(x_2)=0$ if $2\mu_s(x_1) \le \mu_s(x_2)$, $f(y_1)=0$ and $f(x_2)=1$ if $2\mu_s(x_1) > \mu_s(x_2)$, and 
    $f(y_j)=0$, for $j=2,\ldots,\sigma_2$ is an SNRDF on $K_{\sigma_1,\sigma_2}$ which guarantees that $$\gamma_{snR}(K_{\sigma_1,\sigma_2})\le  
            \left\{
                \begin{array}{ll} 
                4\mu_s(x_1),  & \mbox{if } 2\mu_s(x_1) \le \mu_s(x_2) \\[2ex] 
                2\mu_s(x_1)+\mu_s(x_2), & \mbox{if } 2\mu_s(x_1) \ge \mu_s(x_2),
                \end{array}
            \right.
        $$ that is, \begin{equation}\label{ine3}
            \gamma_{snR}(K_{\sigma_1,\sigma_2})\le\min\{4\mu_s(x_1),2\mu_s(x_1)+\mu_s(x_2)\}.
        \end{equation}
    To prove the another inequality, let $f$ be a $\gamma_{snR}$-function on $K_{\sigma_1,\sigma_2}$ with 
    minimum number of vertices in $V_1$. Again, as in $(ii)$, by Remark~\ref{Garifue}, there are no strong 
    neighbors in $V_1$, following that $|V_1|\le 1.$ Clearly, $0<|V_2|\le 2$, because if not, 
    $\gamma_{snR}(K_{\sigma_1,\sigma_2})=w(f)\ge 6\mu_s(x_1)$, contradicting (\ref{ine3}). If $|V_2|=2$ 
    then these two vertices having a label $2$ must belong to different vertex classes, because otherwise, reassigning 
    label 0 to one of them and keeping all other values of $f$ to be the same, we find a new SNRDF on $G$ with 
    a smaller weight than the weight of $f$, a contradiction. Then $\gamma_{snR}(K_{\sigma_1,\sigma_2})=w(f)\ge 
    2\mu_s(x_1)+2\mu_s(y_1)=4\mu_s(x_1)\ge \min\{4\mu_s(x_1),2\mu_s(x_1)+\mu_s(x_2)\}$. If $|V_2|=1$ then one 
    vertex of $X$ must have 1 label by $f$ and the other one must have 2 label, yielding that 
    $\gamma_{snR}(K_{\sigma_1,\sigma_2})=w(f)\ge 2\mu_s(x_1)+\mu_s(x_2)\ge \min\{4\mu_s(x_1),2\mu_s(x_1)+\mu_s(x_2)\}$.
    
    $(iv)$ The function $f:V\rightarrow \{0,1,2\}$ defined as $f(x_1)=2$, $f(y_1)=2$, $f(x_i)=f(y_j)=0$, for 
    $i=2,\ldots,\sigma_1$ and $j=2,\ldots,\sigma_2$ is an SNRDF on $K_{\sigma_1,\sigma_2}$ which guarantees that 
    \begin{equation}\label{ine4} 
        \gamma_{snR}(K_{\sigma_1,\sigma_2})\le 4\mu_s(x_1)
    \end{equation}
    To prove the another inequality, let $f$ be a $\gamma_{snR}$-function on $K_{\sigma_1,\sigma_2}$ with 
    minimum number of vertices in $V_1$. As in $(ii)$ and $(iii)$, by Remark~\ref{Garifue}, there are no strong 
    neighbors in $V_1$. Clearly, $0<|V_2|\le 2$, because if not, $\gamma_{snR}(K_{\sigma_1,\sigma_2})=w(f)\ge 6\mu_s(x_1)$, 
    contradicting (\ref{ine4}). If $|V_2|=2$ then these two vertices with 2 label must belong to different vertex 
    class, because otherwise, reassigning label 0 to one of them and keeping all other values of $f$ to be the same, 
    we find a new SNRDF on $K_{\sigma_1,\sigma_2}$ with a smaller weight than the weight of $f$, which is not 
    possible. Then $\gamma_{snR}(K_{\sigma_1,\sigma_2})=w(f)\ge 2\mu_s(x_1)+2\mu_s(y_1)=4\mu_s(x_1)$. If $|V_2|=1$ 
    then in one the vertex classes, one vertex is labeled with 2 by $f$ and the rest of vertices are labeled with 1. 
    As $\sigma_2\ge \sigma_1\ge 3$, we have $|V_1|\ge 2$, and therefore, 
    $$
    \begin{array}{rcl} 
        \gamma_{snR}(K_{\sigma_1,\sigma_2})=w(f) & \ge & \min\{2\mu_s(x_1)+
                    \mu_s(x_2)+\mu_s(x_3),2\mu_s(y_1)+\mu_s(y_2)+\mu_s(y_3)\} \\[2ex] 
                                                & \ge & 2\mu_s(x_1)+\min\{\mu_s(x_2)+
                    \mu_s(x_3),\mu_s(y_2)+\mu_s(y_3)\}                         \\[2ex] 
                                                & \ge & 2\mu_s(x_1)+\min\{2\mu_s(x_1),2\mu_s(y_1)\}
                                                = 4\mu_s(x_1).
    \end{array}$$ 
This finishes the proof.
\end{proof}

Next we provide a condition that is both necessary and sufficient for a non-trivial fuzzy graph to have 
a $\gamma_{snR}$ twice the size of the graph.

\begin{proposition}
    Let $G:(V,\sigma,\mu)$ be a non-trivial fuzzy graph of size $q$. Then $\gamma_{snR}(G)=2q$ if and only 
    if all edges are strong and each vertex is either an isolated vertex or has a unique strong neighbor. 
\end{proposition}
\begin{proof}
    If all edges are strong and each vertex is either an isolated vertex or has a unique strong neighbor, 
    then $V=X\cup Y\cup Z$, where $X=\{x_1,\ldots,x_r\}$, $Y=\{y_1,\ldots,y_r\}$, $Z=V\setminus (X\cup Y)$ 
    and $\{(x_1,y_1),\ldots,(x_r,y_r)\}$ is the set of strong edges of $G$. 
    Let $f:V\rightarrow \{0,1,2\}$ be the function on $G$ defined as $f(x_i)=2$ and $f(y_i)=0$, for 
    $i=1,\ldots, r$, and $f(w)=1$, if $w\in V\setminus (X\cup Y)$. Clearly, $f$ is an SNRDF and its weight is 
    $w(f)=2\sum_{i=1}^r\mu(x_i,y_i)=2q$. Hence, $\gamma_{snR}(G) \le 2q$. On the other hand, let $f$ be a 
    $\gamma_{snR}(G)$-function. Since $\{x_i,y_i\}$ must mutually dominate each other, then
    $\gamma_{snR}(G)=w(f)\ge 2\sum_{i=1}^r\mu(x_i,y_i)=2q$.
   
    Now, assume that $\gamma_{snR}(G)=2q$. Denote by $E$ and $E_s$ the sets of edges and strong edges of $G$, 
    respectively, and let $D\subseteq V$ be a $\gamma_s$-set of the subgraph of $G$ induced by the strong 
    edges of $G$. Consider the function $f:V\rightarrow \{0,1,2\}$ defined as $f(x)=2$ if $x\in D$, $f(x)=0$ 
    if $x$ is strong dominated by some vertex of $D$, and $f(x)=1$, otherwise, whose weight is 
    \begin{equation}\label{starc} 
    \displaystyle \gamma_{snR}(G)\le w(f)=2\sum_{x\in D}\mu_s(x)\end{equation}
    If there is an edge $(u,v)\in E$ which is not strong, then from (\ref{starc}) if follows that 
    $w(f)\le 2\sum_{e\in E_s}\mu(e)<2\sum_{e\in E}\mu(e)=2q,$ a contradiction. Thus, all edges are strong.
    If there is a vertex with at least two strong neighbors, then $|D|<|E_s|$, yielding from (\ref{starc}) 
    that $\displaystyle \gamma_{snR}(G)\le w(f)<2\sum_{e\in E_s}\mu(e)=2q$, again a contradiction. Hence, we
    have that each vertex either has a unique strong neighbor or is an isolated vertex.
\end{proof}

\begin{remark}\label{20}
Let $f$ be a $\gamma_{snR}$-function on a fuzzy graph $G:(V,\sigma,\mu)$. Then every
vertex $u\in V$ for which $f(u)=2$ has at least one strong neighbor $v\in N_s(u)$ for which $f(v)=0$. 
\end{remark}
\begin{proof}
    If $f(u)=2$ and every $v\in N_s(u)$ has $f(v)\ne 0$, then we can find an 
    SNRDF $g:V\rightarrow \{0,1,2\}$ such that $g(u)=1$, and $g(w)=f(w)$, for every $w\in V\setminus\{u\}$ 
    whose weight is $g(V)=w(f)-\mu_s(u)<w(f)$, being a contradiction to the fact that $f$ is a 
    $\gamma_{snR}$-function. 
\end{proof}

The following lemma shows that from a $\gamma_{snR}$-function on a fuzzy graph with associated partition 
$(V_0,V_1,V_2)$ we can obtain another SNRDF by exchanging the $V_0$ and $V_2$.

\begin{lemma}\label{2change0}
 Let $f:V\rightarrow \{0,1,2\}$ be a $\gamma_{snR}$-function on a fuzzy graph $G:(V,\sigma,\mu)$ with 
 partition $(V_0,V_1,V_2)$. Then the function $g:V\rightarrow \{0,1,2\}$ defined as $g(x)=2$ if $x\in V_0$, 
 $g(x)=1$ if $x\in V_1$, and $g(x)=0$ if $x\in V_2$ is an SNRDF on $G$.    
\end{lemma}
\begin{proof}
    In order to check that $g$ is an SNRDF on $G$ it only remains to prove that every vertex $u\in V$ for
    which $g(u)=0$ has at least one strong neighbor $v\in V$ for which $g(v)=2$. If $g(u)=0$ then $f(u)=2$, 
    and therefore, by Remark~\ref{20}, there exists at least one vertex $v\in N_s(u)$ for which $f(v)=0$. 
    This implies that $g(v)=2$ and the result follows. 
\end{proof}

\begin{remark}\label{V0V2}
    If $f:V\rightarrow \{0,1,2\}$ is a $\gamma_{snR}$-function on a fuzzy graph $G$ with partition $(V_0,V_1,V_2)$, 
    then $|V_2|\le |V_0|$. Further, if $|V_2|=|V_0|$ then every vertex $v\in V_2$ has one only strong neighbor 
    in $V_0$.  
\end{remark}
\begin{proof}
    The proof is straighforward, because if $|V_2|>|V_0|$ then there must exist $u\in V_2$ such that every 
    vertex of $V_0$ has at least one strong neighbor in $V_2\setminus\{u\}$, and therefore, the function 
    $g:V\rightarrow \{0,1,2\}$ defined as $g(u)=1$ and $g(v)=f(v)$, for every $v\in V\setminus\{u\}$ is 
    an SNRDF on $G$ having $w(g)<w(f)$. This contradicts the fact that $f$ is a $\gamma_{snR}$-function. 
    If $|V_2|=|V_0|$ and there is some vertex $w\in V_2$ with at least two strong neighbors in $V_0$, then 
    there must exist $u\in V_2$ such that every vertex of $V_0$ has a strong neighbor in $V_2\setminus\{u\}$. 
    Thus, reasoning as before, we can find an SNRDF $g$ with smaller weight than $f$, a contradiction. 
 \end{proof}

The next result gives a general upper bound for the SNRD number of a fuzzy 
graph in terms of its order. Further, a necessary and sufficient condition to attain equality is given.  

\begin{theorem}\label{thp}
    Let $G:(V,\sigma,\mu)$ be a non-trivial fuzzy graph of order $p$. Then $\gamma_{snR}(G)\le p,$ and the 
    equality holds if and only if $G$ is formed by some isolated vertices and the disjoint union of
    effective edges.
\end{theorem}

\begin{proof}
    Let $f:V\rightarrow \{0,1,2\}$ be a $\gamma_{snR}$-function on $G$ with partition $(V_0,V_1,V_2)$. 
    By Lemma~\ref{2change0} there exists an SNRDF $g:V\rightarrow \{0,1,2\}$ with partition $(V'_0,V'_1,V'_2)$ 
    such that $V'_0=V_2$, $V'_1=V_1$ and $V'_2=V_0$. Then $f(u)+g(u)=2$, for all $u\in V$ and therefore, 
    \begin{equation}\label{fyg}
        \begin{array}{rcl} 
        2\gamma_{snR}(G) & \le & 2w(f)  \le  w(f)+w(g)  
                         =  \displaystyle \sum_{u\in V}  \left(f(u)+g(u)\right) \mu_s(u)  \\[2ex] 
                        & = & \displaystyle 2\sum_{u\in V} \mu_s(u) \le  2\sum_{u\in V} \sigma(u)  = 2p.
        \end{array}
    \end{equation} 
    Hence, $\gamma_{snR}(G)\le p$.

    Now, assume that $\gamma_{snR}(G)=p$. This occurs if and only if all the inequalities of~(\ref{fyg}) 
    become equalities. That is; $g$ is also a $\gamma_{snR}$-function and $\mu_s(u)=\sigma(u)$, for all 
    $u\in V$. As $V_1=V'_1$ and both $f$ and $g$ are $\gamma_{snR}$-functions, by Proposition~\ref{propiedades} 
    there is no edge between $V_1$ and $V_0\cup V_2$, yielding that the vertices of $V_1$, if any, are isolated. 
    Indeed, there are no more isolated vertices than those of $V_1$. Moreover, $|V_2|=|V_0|$, since by 
    Remark~\ref{V0V2}, we have $|V_2|\le |V_0|$ and $|V_0|=|V'_2|\le |V'_0|=|V_2|$. Thus, as $|V_2|=|V_0|$ 
    and therefore, $|V'_2|=|V'_0|$, again by Remark~\ref{V0V2}, each vertex of $V_2$ has one only strong 
    neighbor in $V_0$ and each vertex of $V_0$ has one only strong neighbor in $V_2$. This means that 
    $G[V_0\cup V_2]$ is a set of disjoint edges. Further, each of these edges, $(u_1,v_1),\ldots,(u_r,v_r)$ 
    satisfies that $\sigma(u_i)=\sigma(v_i)$, for $i=1,\ldots,r$, since $f$ and $g$ are $\gamma_{snR}$-functions 
    and $\mu_s(u)=\sigma(u)$, for all $u\in V$. This proves the result.  
\end{proof}

At last, we establish a Nordhaus–Gaddum inequality concerning the combined strong neighbor Roman domination 
numbers of a graph and its complementary graph.

\begin{proposition}
    Let $G$ be a graph with order $p$ such that $G$ and $\overline{G}$ are non-trivial graphs. Then
    $ 2(\mu_{\min}+\overline{\mu}_{\min})\le \gamma_{snR}(G)+\gamma_{snR}(\overline{G})<2p.$
\end{proposition}

\begin{proof}
Let $f$ be a $\gamma_{snR}(G)$-function with minimum number of vertices in $V_1$. 
As $G$ is a non-trivial graph then there exists at least one edge in $G$, and therefore, $G$ has some strong 
edge.
Therefore, $V_2\ne\emptyset$ and $\gamma_{snR}(G)=w(f)\ge 2\mu_s(v)$ for some vertex $v\in V_2.$ So,
$\gamma_{snR}(G)\ge 2\mu_{\min}$. Analogously, $\gamma_{snR}(\overline{G})\ge 2\overline{\mu}_{\min}$ and 
the first inequality holds.

By Theorem \ref{thp} we have $\gamma_{snR}(G)\le p$. As $\overline{G}$ has the same order, we deduce that
$\gamma_{snR}(\overline{G})\le p$. Besides, by Theorem \ref{thp}, the equality holds only when the graph is
formed by disjoint effective edges and isolated vertices. Since both $G$ and $\overline{G}$ could not have 
simultaneously this same structure, we may derive that $\gamma_{snR}(G)+\gamma_{snR}(\overline{G})<2p$.
\end{proof}

\section{The strong-neighbors Roman domination in fuzzy cycles and paths}

Now we focus on approximating the SNRD number in fuzzy paths and fuzzy strong cycles. 
A fuzzy cycle is strong if all its fuzzy edges are strong. Every fuzzy path 
is strong because no other path exists between two neighbors than the edge joining them. 
Further, it is easy to characterize the fuzzy strong cycles, as seen in the following remark. 

\begin{remark}\label{cy2strong}
    Let $n\ge 3$ be an integer and let $C_n:(V,\sigma,\mu)$ be a fuzzy cycle with $|\sigma^*|=n$, 
    $C_n:u_1,u_2,\ldots,u_n,u_1$, so that $\mu(u_n,u_1)=\wedge\{\mu(e):e\in \mu^*\}$. 
    Then $C_n$ is a fuzzy strong cycle if and only if there exists $j\in\{1,\ldots,n-1\}$ such that 
    $\mu(u_j,u_{j+1})=\mu(u_n,u_1)$. 
\end{remark}
\begin{proof}
    First, assume that every edge is strong. As $\mu(u_n,u_1)=\wedge\{\mu(e):e\in \mu^*\}$ and $(u_n,u_1)$ 
    is strong, we have $\mu(u_n,u_1)={CONN}_G(u_n,u_1)$. Thus, $s(P:u_1,u_2,\ldots,u_n)\le \mu(u_n,u_1)$, 
    following the existence of $j\in\{1,\ldots,n-1\}$ such that $\mu(u_j,u_{j+1})=\mu(u_n,u_1)$.

    Second, assume that there is $j\in\{1,\ldots,n-1\}$ such that $\mu(u_j,u_{j+1})=\mu(u_n,u_1)$. As 
    we have two edges with minimum weight, the strength of every path in $G$ is $\mu(u_n,u_1)$, which 
    is less than or equal to the weight of every edge. Hence, every edge is strong. 
\end{proof}

We first prove the following lemma.

\begin{lemma}\label{cotamuscycle}
    Let $n\ge 3$ be an integer and let $C_n:(V,\sigma,\mu)$ be a fuzzy strong cycle with $|\sigma^*|=n$, 
    represented as $C_n:u_1,u_2,\ldots,u_n,u_1$. Then 
    $\sum_{i=1}^n\mu_s(u_i)\le q-\left(\mu_{\max}-\mu_{min}\right)$ and equality holds 
    if and only if there exists $j\in\{1,\ldots , n-2\}$ such that 
    $\mu_{\min}=\mu(u_{j},u_{j+1})=\mu(u_{j+1},u_{j+2})$ and further, 
    $\displaystyle \left\{\mu(u_i,u_{i+1})\right\}_{i=1}^j$ and 
    $\displaystyle \left\{\mu(u_i,u_{(i+1)(\modx n)})\right\}_{i=j+1}^n$ are, respectively, a 
    non-increasing and a non-decreasing sequence.
\end{lemma}
\begin{proof}
    Without loss of generality we may assume that $\mu_{\max}=\mu(u_n,u_1)$. Since $C_n$ is a 
    fuzzy strong cycle, by Remark~\ref{cy2strong} we may assure that there exist at least two 
    edges with minimum weight. Thus, there exists $j\in\{1,\ldots,n-2\}$ such that 
    $\mu_{\min}=\mu(u_j,u_{j+1})$. If $j=1$ then 
    \begin{equation}\label{j1} 
    \begin{array}{rcl}     
        \displaystyle \sum_{i=1}^n\mu_s(u_i) & = & \displaystyle \mu_s(u_1)+\sum_{i=2}^n \mu_s(u_i) 
         \le  \displaystyle  \mu_{\min}+\sum_{i=2}^n\mu(u_{i-1},u_i) \\[2ex] 
        & = & \mu_{\min}+q-\mu(u_n,u_1) 
         =  q-\left(\mu_{\max}-\mu_{min}\right).
    \end{array}
    \end{equation} 
    If $j\ge 2$ then 
\begin{equation}\label{jge2} 
    \begin{array}{rcl}
        \displaystyle \sum_{i=1}^n\mu_s(u_i) & = & \displaystyle\sum_{i=1}^{j}\mu_s(u_i)+\sum_{i=j+1}^n\mu_s(u_i)
     \le   \displaystyle\sum_{i=1}^{j}\mu(u_{i},u_{i+1})+\sum_{i=j+1}^n\mu(u_{i-1},u_i) \\[2.5ex] 
     & = & \displaystyle\sum_{i=1}^{j-1}\mu(u_{i},u_{i+1})+\mu_{\min}+\sum_{i=j+1}^n\mu(u_{i-1},u_i)  \\[2.5ex] 
     & = & \mu_{\min}+q-\mu(u_n,u_1)    =     q-\left(\mu_{\max}-\mu_{min}\right).
    \end{array}
\end{equation}

    If there exists $j\in\{1,\ldots n-2\}$ such that $\mu_{\min}=\mu(u_{j},u_{j+1})$ and further, 
    $\displaystyle \mu(u_1,u_{2})\ge\mu(u_2,u_3)\ge\cdots\ge\mu(u_j,u_{j+1})$ and 
    $\displaystyle \mu(u_{j+1},u_{j+2})\le\mu(u_{j+2},u_{j+3})\le\cdots\le\mu(u_n,u_1)$ then $$\displaystyle 
    \sum_{i=1}^n\mu_s(u_i)=\sum_{i=1}^j\mu(u_i,u_{i+1})+\sum_{i=j+1}^n\mu(u_{i-1},u_{i})=q-\left(\mu_{\max}-\mu_{min}\right).$$ 
     Let us see the converse.
     
     As $C_n$ is a fuzzy strong cycle, we know that there exist two edges with minimum weight. Thus, there exists 
     $j\in\{1,\ldots,n-2\}$ such that $\mu_{\min}=\mu(u_j,u_{j+1})$. If $j=1$ then all the inequalities of~(\ref{j1}) 
     become equalities. That is, $\mu_s(u_1)=\mu(u_1,u_2)$ and $\mu_s(u_i)=\mu(u_{i-1},u_i)$ for $i=2,\ldots,n$, yielding 
     that $\mu_{\min}=\mu(u_1,u_2)\le \mu(u_2,u_3)\le \cdots\le \mu(u_n,u_1)$. As $C_n$ must have at least two edges with 
     minimum weight, it necessarily follows that $\mu_{\min}=\mu(u_1,u_2)= \mu(u_2,u_3)$. Now assume that $j\in\{2,\ldots,n-2\}$ 
     is the minimum integer for which $\mu_s(u_j)=\mu_{\min}$. Again, all the inequalities of~(\ref{jge2}) become equalities. 
     That is, $\mu_s(u_i)=\mu(u_i,u_{i+1})$ for $i=1\ldots,j$, and $\mu_s(u_i)=\mu(u_{i-1},u_i)$ for $i=j+1,\ldots,n$, yielding 
     that $\mu(u_1,u_2)\ge \mu(u_2,u_3)\ge\cdots\ge \mu(u_j,u_{j+1})$  and $\mu_{\min}=\mu(u_j,u_{j+1})\le 
     \mu(u_{j+1},u_{j+2})\le \cdots\le \mu(u_n,u_1)$. Finally, since $C_n$ has at least two edges with minimum weight 
     and $j$ is the minimum integer for which $\mu_s(u_j)=\mu_{\min}$, we deduce that $\mu(u_{j+1},u_{j+2})=\mu_{\min}$ and 
     the result follows.
    
\end{proof}

Lemma~\ref{cotamuscycle} will help us to get an upper bound for the strong-neighbors Roman domination number of a strong 
cycle, which is attained, as we can see in the following results.

\begin{theorem}\label{cyclecota}
    Let $n,k\ge 1$ be integers with $k=\lfloor n/3\rfloor$. Let $C_n=(V,\sigma,\mu)$ be a fuzzy cycle with 
    $|\sigma^*|=n, C_n:u_1,u_2,\ldots,u_n,u_1$, with at least two edges with minimum weight. Then 
    $$\displaystyle\gamma_{snR}(C_n)\le \left(1-\frac{k}{n}\right)\left(q-(\mu_{\max}-\mu_{\min})\right).$$
\end{theorem}
\begin{proof}
    By Remark~\ref{cy2strong}, it follows that all the edges of $C_n$ are strong.
We will distinguish three possible cases according to the remainder of the quotient $n/k$.

{\it Case 1.} If $n=3k$. For every $m=1\ldots,n$, let $f_m:V\rightarrow \{0,1,2\}$ be the function defined as  
$f_{m}(u_{(m+3i)(\modx n)})=2$ for $i=0,\ldots,k-1$ and $f_m(v)=0$, otherwise. Clearly, $\{ f_m:m=1,\ldots,n\}$ are SNRDFs 
on $C_n$, and from Lemma~\ref{cotamuscycle} it follows that  
    \begin{equation}\label{cyine0}
            n\cdot\gamma_{snR}(C_n)\le \displaystyle\sum_{m=1}^n w(f_m) =  \displaystyle  2k\sum_{j=1}^{n}\mu_s(u_{j})   
                                   \le  \displaystyle 2k\left(q-(\mu_{\max}-\mu_{\min})\right),
    \end{equation} 
following that $\displaystyle \gamma_{snR}(C_n)\le \frac{2k}{n}\left(q-(\mu_{\max}-\mu_{\min})\right)=
\left(1-\frac{k}{n}\right)\left(q-(\mu_{\max}-\mu_{\min})\right)$.

{\it Case 2.} If $n=3k+1$. For every $m=1\ldots,n$, let $f_m:V\rightarrow \{0,1,2\}$ be the function defined as  
$f_m(u_{m})=1$, $f_{m}(u_{(m+2+3i)(\modx n)})=2$ for $i=0,\ldots,k-1$ and $f_m(v)=0$, otherwise. As  
$f_m$ are SNRDFs on $C_n$ for all $m$, and from Lemma~\ref{cotamuscycle} it follows that  
    \begin{equation}\label{cyine1}
            n\cdot\gamma_{snR}(C_n)\le \displaystyle\sum_{m=1}^n w(f_m) =  \displaystyle  (2k+1)\sum_{j=1}^{n}\mu_s(u_{j})   
                                       \le  \displaystyle (2k+1)\left(q-(\mu_{\max}-\mu_{\min})\right)
    \end{equation} 
following that $\displaystyle \gamma_{snR}(C_n)\le \frac{2k+1}{n}\left(q-(\mu_{\max}-\mu_{\min})\right)=
\left(1-\frac{k}{n}\right)\left(q-(\mu_{\max}-\mu_{\min})\right)$.

{\it Case 3.} If $n=3k+2$. For every $m=1\ldots,n$, let $f_m:V\rightarrow \{0,1,2\}$ be the function defined as  
$f_{m}(u_{(m+3i)(\modx n)})=2$ for $i=0,\ldots,k$ and $f_m(v)=0$, otherwise. Clearly, $\{f_m: m=1,\ldots,n\}$ 
is a set of SNRDFs on $C_n$, and from Lemma~\ref{cotamuscycle} it follows that  
    \begin{equation}\label{cyine2}
            n\cdot\gamma_{snR}(C_n)\le \displaystyle\sum_{m=1}^n w(f_m) =  \displaystyle  2(k+1)
                    \sum_{j=1}^{n}\mu_s(u_{j})  \le  \displaystyle (2k+2)\left(q-(\mu_{\max}-\mu_{\min})\right)
    \end{equation} 
following that $\displaystyle \gamma_{snR}(C_n)\le \frac{2k+2}{n}\left(q-(\mu_{\max}-\mu_{\min})\right)=
\left(1-\frac{k}{n}\right)\left(q-(\mu_{\max}-\mu_{\min})\right)$.
\end{proof}

\begin{remark}\label{cuentas}
Let $n\ge 4$, $k\ge 1$ be two integers such that $n=3k+1$. Let $C_n=(V,\sigma,\mu)$ be a fuzzy strong cycle with 
$|\sigma^*|=n$, represented as $C_n:u_1,u_2,\ldots,u_n,u_1$. For every $m=1\ldots,n$, let $f_m$ 
be the function defined as $f_m(u_{m})=1$, $f_{m}(u_{(m+2+3i)(\modx n)})=2$ for $i=0,\ldots,k-1$ and $f_m(v)=0$, 
otherwise. Then $f_m$ are SNRDFs on $C_n$, for all $m$, and $w(f_m)-w(f_{(m+3) \, (\modx n)})=
\mu_s(u_m)-2\mu_s(u_{(m+1) \, (\modx n)})+2\mu_s(u_{(m+2)\, (\modx n)})-\mu_s(u_{(m+3) \, (\modx n)}).$
\end{remark}
\begin{proof}
    From now on, all subscripts will be considered modulo $n$. It is straightforward to check that $\{f_m:m=1,\ldots,n\}$ is a set of SNRDFs of $C_n$, because for every 
    $i=0,\ldots,k-1$, we have that $u_{m+1+3i}$ 
    and $u_{m+3+3i}$ are strong neighbors of $u_{m+2+3i}$. Moreover, 
    $$
        \begin{array}{rcl}
    w(f_m)-w(f_{m+3}) & = & \displaystyle \mu_s(u_m)+\sum_{i=0}^{k-1} 2\mu_s(u_{m+2+3i})
     \displaystyle -\mu_s(u_{m+3})-\sum_{i=0}^{k-1} 2\mu_s(u_{m+5+3i}) \\[2.5ex] 
    & = & \displaystyle \mu_s(u_m)+\sum_{i=0}^{k-1} 2\mu_s(u_{m+2+3i})           
     \displaystyle -\mu_s(u_{m+3})-\sum_{i=1}^{k} 2\mu_s(u_{m+2+3i})   \\[2.5ex] 
    & = & \mu_s(u_m)-\mu_s(u_{m+3})+2\mu_s(u_{m+2})   -2\mu_s(u_{m+1}).
    \end{array}$$  
\end{proof}

\begin{theorem}\label{cycle:n_not_3k}
    Let $n,k\ge 1$ be integers such that $n=3k+\ell$ with $\ell\in\{1,2\}$. Let $C_n=(V,\sigma,\mu)$ be 
    a fuzzy cycle with $|\sigma^*|=n$, represented as $C_n:u_1,u_2,\ldots,u_n,u_1$, with at least two edges with 
    minimum weight. Then $\displaystyle\gamma_{snR}(C_n)=\left(1-\frac{k}{n}\right)\left(q-(\mu_{\max}-\mu_{\min})\right)$ 
    if and only if at least $n-1$ edges have weight equal to $\mu_{\min}$ and at most one only edge has 
    weight equal to $\mu_{\max}$.
\end{theorem}
\begin{proof}
   By Remark~\ref{cy2strong}, we know that $C_n$ is a fuzzy strong cycle. Without loss of generality, assume that 
   $\mu_{\max}=\mu(u_n,u_1)$.

   If at least $n-1$ edges have weight $\mu_{\min}$ and at most only one edge has weight $\mu_{\max}$, then 
   $\mu_s(u_i)=\mu_{\min}$, for $i=1,\ldots,n$ and $q=(n-1)\mu_{\min}+\mu_{\max}$, which implies that 
   $q-(\mu_{\max}-\mu_{\min})=n\mu_{\min}$. Thus, we only need to prove that $\gamma_{snR}(C_n)\ge 
   \left(1-\frac{k}{n}\right)n\mu_{\min}=(2k+\ell)\mu_{\min}$, because the other inequality holds by 
   Theorem~\ref{cyclecota}. Let $f$ be a $\gamma_{snR}(C_n)$-function with partition $(V_0,V_1,V_2)$. 
   Observe that 
   \begin{equation}\label{V2'} 
   \displaystyle 
        |V_2| \ge \left\lceil \frac{n-|V_1|}{3}\right\rceil 
            = k+1-\left\lfloor \frac{|V_1|+3-\ell}{3}\right\rfloor.
    \end{equation} 
If $|V_1|=0$ then from~(\ref{V2'}) it follows that $|V_2|\ge k+1$.
If $|V_1|=1$ and $\ell=2$ then again by~(\ref{V2'}) we have that $|V_2|\ge k+1$. So,
    $\gamma_{snR}(C_n)\ge 2\mu_{\min}|V_2|\ge 2(k+1)\mu_{\min}\ge (2k+1)\mu_{\min}$. 

\noindent If $|V_1|=1$ and $\ell=1$ then we may derive from~(\ref{V2'}) that $|V_2|\ge k$ and therefore
    $\gamma_{snR}(C_n)\ge 2\mu_{\min}|V_2|+\mu_{\min}|V_1|\ge 2k\mu_{\min}+\mu_{\min} = (2k+1)\mu_{\min}$. 
    
\noindent If $|V_1|\ge 2$, again from~(\ref{V2'}) we have   
   $$
   \begin{array}{rcl}
        \gamma_{snR}(C_n) & = & \displaystyle\sum_{u\in V_1}\mu_s(u)+2\sum_{u\in V_2}\mu_s(u)
                        = \left(|V_1|+2|V_2|\right)\mu_{\min}\\[2ex] 
                        & = & \displaystyle \left(|V_1|+2\left(k+1-\left\lfloor 
                        \frac{|V_1|+3-\ell}{3}\right\rfloor \right) \right)\mu_{\min}\\[2.5ex] 
                        & \ge & \displaystyle\left(2k+2+|V_1|-\frac{2}{3}(|V_1|+3-\ell) \right)\mu_{\min}\\[2.5ex] 
                        & = & \displaystyle\left(2k+\ell +\frac{|V_1|-\ell}{3} \right)\mu_{\min}
                         \ge  \displaystyle\left(2k+\ell \right)\mu_{\min}
    \end{array}$$
    because $|V_1|\ge 2\ge \ell.$ 
    
Now, let us show the converse. Asume that $\displaystyle\gamma_{snR}(C_n)=\left(1-\frac{k}{n}\right)\left(q-(\mu_{\max}-\mu_{\min})\right)$. 
First, suppose that $n=1 (\modx{3})$.  Then all the inequalities of~(\ref{cyine1}) in 
   Theorem~\ref{cyclecota} become equalities. In particular, all the functions of $\{f_m: m=1,\ldots,n\}$ are 
   $\gamma_{snR}(C_n)$-functions. Taking into account Lemma~\ref{cotamuscycle}, if we prove 
   that $\mu_s(u_j)=\mu_s(u_{(j+1)(\modx{n})})$ for all $j=1,\ldots,n$, then we are done. 
   From now on, all subscripts will be considered modulo $n$. We will distinguish two cases:

   {\it Case 1.} If $k$ is even. As $\{f_m:m=1,\ldots,n\}$ is a set of $\gamma_{snR}(C_n)$-functions, 
   all the functions have the same weight, $\gamma_{snR}(C_n)$, and we
   have 
   \begin{equation}\label{sumceros}
    \begin{array}{c}
        \displaystyle w(f_j)-w(f_{j+1}) -2\sum_{r=0}^{k/2-1}\Bigl( w(f_{j+2+6r})-
                    w(f_{j+5+6r})+w(f_{j+3+6r})-w(f_{j+6+6r}) \Bigr)   =  0,
    \end{array}
    \end{equation}
By considering Remark~\ref{cuentas} we also know that  
$$
    \begin{array}{lcl} 
        w(f_{j+2+6r})-w(f_{j+5+6r})  & = & \mu_s(u_{j+2+6r})-2\mu_s(u_{j+3+6r})+
                        2\mu_s(u_{j+4+6r})-\mu_s(u_{j+5+6r})                       \\[1.5ex] 
        w(f_{j+3+6r})-w(f_{j+6+6r})  & = & \mu_s(u_{j+3+6r})-2\mu_s(u_{j+4+6r})+
                        2\mu_s(u_{j+5+6r})-\mu_s(u_{j+6+6r}) 
    \end{array}
$$ 
Therefore, the equation~(\ref{sumceros}) may be reduced to
\begin{equation}\label{sumceros2}
    \begin{array}{c}
        \displaystyle w(f_j)-w(f_{j+1}) -2\sum_{r=0}^{k/2-1}
        \Bigl( \mu_s(u_{j+2+6r})-\mu_s(u_{j+3+6r})+\mu_s(u_{j+5+6r})-\mu_s(u_{j+6+6r}) \Bigr)   =  0,
    \end{array}
    \end{equation}
Besides, if we rewrite the summands in this way,
$$
    \begin{array}{c}
        \mu_s(u_{j+2+3(2r)})-\mu_s(u_{j+3+3(2r)})+\mu_s(u_{j+2+3(2r+1)})-\mu_s(u_{j+3+3(2r+1)}),
    \end{array}
$$
from~(\ref{sumceros2}) we can deduce
\begin{equation}\label{sumceros3}
    \begin{array}{c}
        \displaystyle w(f_j)-w(f_{j+1}) -2\sum_{i=0}^{k-1}\Bigl(\mu_s(u_{j+2+3i})-\mu_s(u_{j+3+3i})\Bigr) = 0 \\[2.5ex]
    \end{array}
\end{equation}
Finally, we have to keep in mind that
$ w(f_j)=\mu_s(u_j)+2\sum_{i=0}^{k-1} \mu_s(u_{j+2+3i})$
for all $j$. So, from~(\ref{sumceros3}) we derive 
$ \mu_s(u_j)-\mu_s(u_{j+1})=0,$
yielding that $\mu_s(u_j)=\mu_s(u_{j+1})$.

    {\it Case 2.} If $k$ is odd.  As $\{f_m:m=1,\ldots,n\}$ is a set of $\gamma_{snR}(C_n)$-functions, 
   all the functions have the same weight, $\gamma_{snR}(C_n)$, and therefore the following equalities hold 
   \begin{equation}\label{sumceros4}
    \begin{array}{rcl}
        -2 \Bigl( w(f_{j-2})-w(f_{j-1})\Bigr) + w(f_{j})-w(f_{j+1})  & =  & 0 \\[2.5ex]
        \displaystyle \sum_{r=0}^{\frac{k-1}{2}-1} \Bigl( -2 \left( w(f_{j+2+6r})-w(f_{j+5+6r})\right)
                        +2\left(w(f_{j+3+6r})-w(f_{j+6+6r}\right) \Bigr) & = & 0 \\[2.5ex]
        \displaystyle \sum_{r=0}^{\frac{k-1}{2}-1} 4 \Bigl( w(f_{j+4+6r})-w(f_{j+7+6r}) \Bigr)  & = & 0
    \end{array}
    \end{equation}
By considering Remark~\ref{cuentas} we also know that  
\begin{equation}\label{delocos}
    \begin{array}{c} 
        w(f_{j+2+6r})-w(f_{j+5+6r})   = \mu_s(u_{j+2+6r})-2\mu_s(u_{j+3+6r})+
                        2\mu_s(u_{j+4+6r})-\mu_s(u_{j+5+6r})                       \\[1.5ex] 
        w(f_{j+3+6r})-w(f_{j+6+6r})   = \mu_s(u_{j+3+6r})-2\mu_s(u_{j+4+6r})+
                        2\mu_s(u_{j+5+6r})-\mu_s(u_{j+6+6r})  \\[1.5ex] 
        w(f_{j+4+6r})-w(f_{j+7+6r})   = \mu_s(u_{j+4+6r})-2\mu_s(u_{j+5+6r})+
                        2\mu_s(u_{j+6+6r})-\mu_s(u_{j+7+6r})  
    \end{array}
\end{equation} 
Therefore, adding the terms on the left-hand side of the three equations~(\ref{sumceros4}) and 
substituting the expressions~(\ref{delocos}) results in the following.
\begin{equation}\label{delocos2}
    \begin{array}{rcl} 
         0  &  =  & -2 \Bigl( w(f_{j-2})-w(f_{j-1})\Bigr) + w(f_{j})-w(f_{j+1}) + \\[2.5ex] 
         & & \displaystyle \sum_{i=0}^{k-2}\left(-2\mu_s(u_{j+2+3i})+6\mu_s(u_{j+3+3i})-4\mu_s(u_{j+4+3i})\right) 
    \end{array}    
\end{equation}
Finally, it is known that 
$ w(f_j)=\mu_s(u_j)+2\sum_{i=0}^{k-1} \mu_s(u_{j+2+3i})$
for each possible value of $j$. So, the equation~(\ref{delocos2}) can be easily reduced to  
$$
    \begin{array}{rl} 
     0 = & \displaystyle -2\mu_s(u_{j-2})-4\mu_s(u_{j})-\sum_{i=0}^{k-2}4\mu_s(u_{j+3+3i})+2\mu_s(u_{j-1})+4\mu_s(u_{j+1})+
                        \sum_{i=0}^{k-2}4\mu_s(u_{j+4+3i})   \\[2.5ex] 
         &  \displaystyle +\mu_s(u_{j})+2\mu_s(u_{j-2})+ \sum_{i=0}^{k-2}2\mu_s(u_{j+2+3i})-\mu_s(u_{j+1})-
       2\mu_s(u_{j-1})-\sum_{i=0}^{k-2}2\mu_s(u_{j+3+3i})  +   \\[2.5ex] 
         &  \displaystyle \sum_{i=0}^{k-2}\left(-2\mu_s(u_{j+2+3i})+6\mu_s(u_{j+3+3i})-4\mu_s(u_{j+4+3i})\right)  \\[3ex]
       = & -3\mu_s(u_j)+3\mu_s(u_{j+1}) ,
     \end{array}
$$ 
yielding that $\mu_s(u_j)=\mu_s(u_{j+1})$, which finishes the proof.

\end{proof}

We may think that one of the functions defined in the proof of Theorem~\ref{cyclecota} could lead us to the exact 
value of the strong-neighbors Roman domination number of a cycle. However, that is not always the case. The influence 
of the weight of the strong edges is such that it is difficult to find a valid strategy to obtain the strong-neighbors 
Roman domination number of any cycle. In Figure~\ref{C6Algeciras} we can observe an example illustrating the described 
situation. The strong-neighbors Roman domination number of the cycle $C_6$ is equal to $0.34$, and it is obtained by 
defining the function $f:V\rightarrow \{0,1,2\}$ as $f(u_1)=f(u_5)=2$, $f(u_3)=1$ and $f(u_2)=f(u_4)=f(u_6)=0$. This 
function has a weight $w(f)=0.34$, whereas the functions $f_m$,  defined as in Theorem~\ref{cyclecota} have weights 
$w(f_m)=0.62$, $m=1,\ldots ,6$.

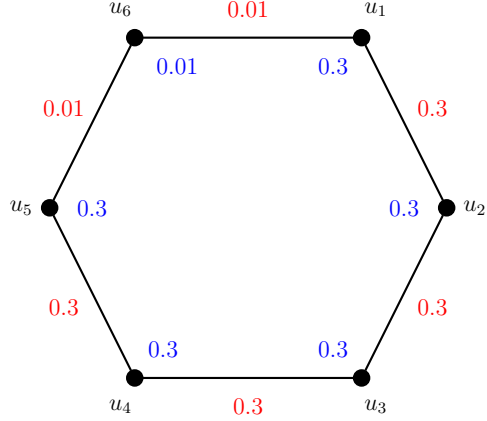
\begin{figure}[ht]
\centering
\begin{tikzpicture}[scale=.75, transform shape, style=thick]
% Node styles
\tikzstyle{vert_black}=[fill=black, draw=black, shape=circle, scale=0.75]
\tikzstyle{none}=[fill=none, draw=none, shape=circle]
\tikzstyle{vert_dest}=[fill=none, draw=red, shape=circle, scale=2]
	
		\node [style={vert_black}] (0) at (-4, 6) {};
		\node [style={vert_black}] (1) at (0, 6) {};
		\node [style={vert_black}] (2) at (-5.5, 3) {};
		\node [style={vert_black}] (3) at (-4, 0) {};
		\node [style={vert_black}] (4) at (0, 0) {};
		\node [style={vert_black}] (5) at (1.5, 3) {};
		\node [style=none] (6) at (-3.25, 5.5) {\color{blue} \large $0.01$};
		\node [style=none] (7) at (-0.5, 5.5) {\color{blue} \large $0.3$};
		\node [style=none] (8) at (0.75, 3) {\color{blue} \large $0.3$};
		\node [style=none] (9) at (-0.5, 0.5) {\color{blue} \large $0.3$};
		\node [style=none] (10) at (-3.5, 0.5) {\color{blue} \large $0.3$};
		\node [style=none] (11) at (-4.75, 3) {\color{blue} \large $0.3$};
		\node [style=none] (12) at (-2, 6.5) {\color{red} \large $0.01$};
		\node [style=none] (13) at (1.25, 4.75) {\color{red} \large $0.3$};
		\node [style=none] (14) at (1.25, 1.25) {\color{red} \large $0.3$};
		\node [style=none] (15) at (-2, -0.5) {\color{red} \large $0.3$};
		\node [style=none] (16) at (-5.25, 1.25) {\color{red} \large $0.3$};
		\node [style=none] (17) at (-5.25, 4.75) {\color{red} \large $0.01$};
		\node [style=none] (18) at (0.25, 6.5) {\large $u_1$};
		\node [style=none] (19) at (2, 3) {\large $u_2$};
		\node [style=none] (20) at (0.25, -0.5) {\large $u_3$};
		\node [style=none] (21) at (-4.25, -0.5) {\large $u_4$};
		\node [style=none] (22) at (-6, 3) {\large $u_5$};
		\node [style=none] (23) at (-4.25, 6.5) {\large $u_6$};

            \draw (5) to (4);
		\draw (4) to (3);
		\draw (3) to (2);
		\draw (2) to (0);
		\draw (0) to (1);
		\draw (1) to (5);
\end{tikzpicture}
   \vspace{-.25cm} \caption{A cycle for which the functions defined in Th.~\ref{cyclecota} are not optimal. }\label{C6Algeciras}
\end{figure}

The following result provides an upper bound for the SNRD number of fuzzy paths.

\begin{theorem}\label{pathcota}
    Let $n\ge 6$ be an integer and $P_n=(V,\sigma,\mu)$ be a fuzzy path with $|\sigma^*|=n$, 
    represented as $P_n:u_1,u_2,\ldots,u_n.$ 
    Then $\gamma_{snR}(P_n)\le \frac{2}{3}\left(q+\mu_{\min}+\mu_{\max}\right).$
\end{theorem}
\begin{proof}
Recall that all the edges of every path are strong. When all the edges of $P_n$ have the same weight,
studying the strong-neighbors Roman domination number of a fuzzy path $P_n$ with $|\sigma^*|=n$ is
equivalent to studying the Roman domination number of a path $P_n$ with $n$ vertices and non-weighted edges. 
In this last case, it was proven in~\cite{CDHH04} that $\gamma_R(P_n)=\left\lceil\frac{2n}{3}\right\rceil$. 
Thus, as the size of a fuzzy path $P_n$ with $|\sigma^*|=n$ and all its edges with the same weight, 
$\mu_{\min}$, is $q=(n-1)\mu_{\min}$, we have 
$$\gamma_{snR}(P_n)=\left\lceil\frac{2n}{3}\right\rceil\mu_{\min}\le \frac{2n+2}{3}\mu_{\min}=
\frac{2}{3}(q+2\mu_{\min})=\frac{2}{3}(q+\mu_{\min}+\mu_{\max}).$$ Then, in the rest of the proof 
we assume that $\mu_{\min}<\mu_{\max}$. Let $\ell,\ell^*\in\{1,\ldots,n-1\}$ be two integers such that 
$\mu(u_{\ell},u_{\ell+1})=\mu_{\min}$ and $\mu(u_{\ell^*},u_{\ell^*+1})=\mu_{\max}$.  As $P_n$ is 
an undirected fuzzy graph, without loss of generality we may assume that $\ell<\ell^*$. 

Denote by $k=\lfloor n/3\rfloor$. Let $f_1:V\rightarrow \{0,1,2\}$ be the function defined as 
$f_1(u_{1+3i})=2$  for $i=0,\ldots,k-1,$ $f_1({u_n})=1$ when $n=0 \, (\modx 3),$ $f_1(u_{n-1})=2$
when $n=1,2\,  (\modx 3)$ and $f_1(v)=0$ otherwise. 
Let $f_2:V\rightarrow \{0,1,2\}$ be the function defined as $f_2(u_{2+3i})=2$  for 
$i=0,\ldots,k-1,$ $f_2(u_{n})=1$ when $n=1 \, (\modx 3),$ $f_2(u_{n-1})=2$ when $n=2\, (\modx 3)$ 
and $f_2(v)=0$ otherwise.
And let $f_3:V\rightarrow \{0,1,2\}$ be the function defined as $f_3(u_{3+3i})=2$  for 
$i=0,\ldots,k-2,$ $f_3({u_1})=1$ when $n=0,1,2 \, (\modx 3),$ $f_3({u_{n-1}})=2$ when $n=0,1 \, (\modx 3),$ 
$f_3({u_{n-2}})=2$ when $n=2 \, (\modx 3),$ $f_3({u_n})=1$ when $n=2 \, (\modx 3)$ and $f_3(v)=0$ otherwise.
Clearly, $\{f_m: m=1,\ldots,3\}$ is a set of SNRDFs on $P_n$. If $\ell^*=n-1$ then
\begin{equation}\label{sum3fmn-1}
\begin{array}{rcl}
    \displaystyle 3\cdot\gamma_{snR}(P_n) & \le & \displaystyle \sum_{m=1}^3 w(f_m) 
        =  \displaystyle
    \mu_s(u_1)+\sum_{i=1}^{n-1} 2\mu_s(u_{i})+2\mu_s(u_{n-1})+\mu_s(u_n)\\[2.5ex] 
        & = & \displaystyle \mu_s(u_1)+\sum_{i=1}^{\ell} 2\mu_s(u_{i})+
    \sum_{i=\ell+1}^{n-1} 2\mu_s(u_{i})+2\mu_s(u_{n-1})+\mu_s(u_n)\\[2.5ex] 
        & \le & \displaystyle \sum_{i=1}^{\ell} 2\mu(u_{i},u_{i+1})+\sum_{i=\ell+1}^{n-1}2\mu(u_{i-1},u_i)+
    4\mu_{\max}\\[2.5ex] & = & 2q-2\mu_{\max}+2\mu_{\min}+4\mu_{\max}
         =  2\left(q+\mu_{\min}+\mu_{\max}\right).
\end{array}
\end{equation}
If $\ell^*\le n-2$ then, reasoning as before, we have \begin{equation}\label{sum3fm}
\begin{array}{rcl}
    \displaystyle 3\cdot\gamma_{snR}(P_n) & \le & \displaystyle \sum_{m=1}^3 w(f_m) 
         =  \displaystyle \mu_s(u_1)+\sum_{i=1}^{n-1} 2\mu_s(u_{i})+2\mu_s(u_{n-1})+\mu_s(u_n)\\[2.5ex] 
        & \le & \displaystyle \sum_{i=1}^{\ell} 2\mu(u_{i},u_{i+1})+\sum_{i=\ell+1}^{\ell^*}2
        \mu(u_{i-1},u_i)+\sum_{i=\ell^*+1}^{n-1} 2\mu(u_{i},u_{i+1}) +4\mu_{\max}\\[2.5ex] 
        & = & 2q-2\mu_{\max}+2\mu_{\min}+4\mu_{\max}
         =  2\left(q+\mu_{\min}+\mu_{\max}\right).
\end{array}
\end{equation}
This finishes the proof.
\end{proof}

In the next lemma we identify certain characteristics that a fuzzy path must satisfy 
when the inequality presented in the previous theorem is, in fact, an equality.

\begin{lemma}\label{cotamuspath}
    Let $n\ge 3$ be an integer and let $P_n:(V,\sigma,\mu)$ be a fuzzy  path 
    with $|\sigma^*|=n$, represented as $P_n:u_1,u_2,\ldots,u_n$. The following assertions hold,
    \begin{itemize}
        \item[(i)] $\displaystyle \sum_{i=1}^n\mu_s(u_i)\le q+\mu_{min}.$
        \item[(ii)] $\displaystyle \sum_{i=1}^n\mu_s(u_i)= q+\mu_{min}$ if and only if 
        there exist $\ell,\ell^*\in\{1,\ldots,n-1\}$ such that \linebreak 
        $\mu_{\min}=\mu(u_{\ell},u_{\ell+1})$ and $\mu_{\max}=\mu(u_{\ell^*},u_{\ell^*+1})$, 
        and further, $\displaystyle \left\{\mu(u_i,u_{i+1})\right\}_{i=1}^\ell$ and 
        \linebreak $\displaystyle \left\{\mu(u_i,u_{(i+1)})\right\}_{i=\ell^*}^{n-1}$ 
        are non-increasing sequences, and 
        $\displaystyle \left\{\mu(u_i,u_{(i+1)})\right\}_{i=\ell}^{\ell^*}$ 
        is a non-decrea\-sing sequence.  
    \end{itemize} 
\end{lemma}

\begin{proof}
$(i)$    Let $\ell,\ell^*\in\{1,\ldots,n-1\}$ be integers such that 
$\mu_{\min}=\mu(u_\ell,u_{\ell+1})$ and $\mu_{\max}=\mu(u_{\ell^*},u_{\ell^*+1})$. 
Without loss of generality we may assume that $\ell<\ell^*$. \newline
If $\ell^*=n-1$ then 
\begin{equation}\label{jp1} 
    \begin{array}{rcl}
        \displaystyle \sum_{i=1}^{n}\mu_s(u_i) & = & \displaystyle\sum_{i=1}^{\ell}
            \mu_s(u_i)+\sum_{i=\ell+1}^{n}\mu_s(u_i)      \\[2.5ex] 
        & \le &  \displaystyle\sum_{i=1}^{\ell} \mu(u_i,u_{i+1})+
            \sum_{i=\ell+1}^{n}\mu(u_{i-1},u_i)      
             =  q+\mu_{min}.
    \end{array}
\end{equation} 
If $\ell^*\le n-2$ then 
\begin{equation}\label{jpge2} 
    \begin{array}{rcl}
    \displaystyle \sum_{i=1}^{n}\mu_s(u_i) & = & 
    \displaystyle\sum_{i=1}^{\ell}\mu_s(u_i)+\sum_{i=\ell+1}^{\ell^*}\mu_s(u_i)+
                    \sum_{i=\ell^*+1}^n\mu_s(u_i)           \\[2.5ex] 
    & \le &  \displaystyle\sum_{i=1}^{\ell}
    \mu(u_i,u_{i+1})+\sum_{i=\ell+1}^{\ell^*}\mu(u_{i-1},u_i)+
    \sum_{i=\ell^*+1}^{n-1}\mu(u_i,u_{i+1})+\mu_{\max}     % \\[2.5ex] 
     =  q+\mu_{min},
    \end{array}
\end{equation}
which conclude the proof of item $(i)$.

$(ii)$ Suppose that $\displaystyle \sum_{i=1}^n\mu_s(u_i)= q+\mu_{min}$. 

If $\ell^*=n-1$ then all the inequalities of~(\ref{jp1}) become equalities. 
Thus, from $\mu_s(u_i)=\mu(u_i,u_{i+1})$ for $1\le i \le \ell$, it follows that 
$\mu(u_{i-1},u_{i})\ge \min\{\mu(u_{i-1},u_{i}), \mu(u_i,u_{i+1})\} = \mu_s(u_i)=  
\mu(u_i,u_{i+1})$, for $2\le i\le \ell$. And also from $\mu_s(u_i)=\mu(u_{i-1},u_{i})$ 
for $\ell+1\le i \le n$, it follows that $\mu(u_{i-1},u_{i})=\mu_s(u_i)=
\min\{\mu(u_{i-1},u_{i}), \mu(u_i,u_{i+1})\}\le \mu(u_{i},u_{i+1})$, 
for $\ell+1\le i\le n-1$.

If $\ell^*\le n-2,$ similarly to the previous reasoning, all the inequalities 
of~(\ref{jpge2}) become equalities and therefore, from $\mu_s(u_i)=\mu(u_i,u_{i+1})$ 
for $1\le i \le \ell$ and $\ell^*+1\le i\le n-1$, it follows that 
$\mu(u_{i-1},u_{i})\ge \min\{\mu(u_{i-1},u_{i}), \mu(u_i,u_{i+1})\} = \mu_s(u_i)= 
\mu(u_i,u_{i+1})$ for $2\le i\le \ell$ and $\ell^*+1\le i\le n-1$. And from 
$\mu_s(u_i)=\mu(u_{i-1},u_{i})$ for $\ell+1\le i \le \ell^*$, it follows that
$\mu(u_{i-1},u_{i})=\mu_s(u_i)=\min\{\mu(u_{i-1},u_{i}), \mu(u_i,u_{i+1})\}\le
\mu(u_{i},u_{i+1})$, for $\ell+1\le i\le \ell^*.$ This finishes the proof.
\end{proof}

\begin{theorem}
Let $n\ge 3$ be an integer and let $P_n:(V,\sigma,\mu)$ be a fuzzy strong path
with $|\sigma^*|=n$, represented as $P_n:u_1,u_2,\ldots,u_n$. If 
$\gamma_{snR}(P_n) = \frac{2}{3}\left(q+\mu_{\min}+\mu_{\max}\right)$ then the 
following assertions hold,
\begin{itemize}
   % \item[(i)] $\displaystyle \sum_{i=1}^n\mu_s(u_i)\le q+\mu_{min}.$
    \item[(i)] There exist $\ell,\ell^*\in\{1,\ldots,n-1\}$ such that 
    $\mu_{\min}=\mu(u_{\ell},u_{\ell+1})$ and $\mu_{\max}=\mu(u_{\ell^*},u_{\ell^*+1})$, 
    and further, $\displaystyle 
    \left\{\mu(u_i,u_{i+1})\right\}_{i=1}^\ell$ and $\displaystyle 
    \left\{\mu(u_i,u_{(i+1)})\right\}_{i=\ell^*}^{n-1}$ are non-increasing 
    sequences, and 
    \linebreak $\displaystyle \left\{\mu(u_i,u_{(i+1)})\right\}_{i=\ell}^{\ell^*}$ is 
    a non-decreasing sequence.
    \item[(ii)] $\mu_s(u_1) = \mu_s(u_2) = \mu_s(u_{n-1}) = \mu_s(u_{n})=\mu_{\max}.$
\end{itemize}
\end{theorem}

\begin{proof}
Reasoning as in Theorem~\ref{pathcota}, we obtain

$$
\begin{array}{rcl}
    \displaystyle 3\gamma_{snR}(P_n) & \le & 
    \displaystyle\sum_{i=1}^{n-1} 2\mu_s(u_{i})+\mu_s(u_1)+2\mu_s(u_{n-1})+
    \mu_s(u_n) \\[2.ex]
    &= & \displaystyle 2\sum_{i=1}^{n} \mu_s(u_{i})+\mu_s(u_1)+
    2\mu_s(u_{n-1})-\mu_s(u_n). 
\end{array}
$$

Thus,

$3\gamma_{snR}(P_n) =2\left(q+\mu_{\min}+\mu_{\max}\right) \le
\displaystyle 2\sum_{i=1}^{n} \mu_s(u_{i})+\mu_s(u_1)+2\mu_s(u_{n-1})-\mu_s(u_n)$ 

yielding that  
\begin{equation}\label{pn}
2\sum_{i=1}^{n} \mu_s(u_{i})\ge 2\left(q+\mu_{\min}+\mu_{\max}\right)
-\mu_s(u_1)-2\mu_s(u_{n-1})+\mu_s(u_n). 
\end{equation}

Now, we consider the fuzzy strong path $P^*_n: u_n,u_{n-1},\ldots,u_1.$ 
Clearly, $\gamma_{snR}(P^*_n) =\gamma_{snR}(P_n) = \frac{2}{3}\left(q+\mu_{\min}+\mu_{\max}\right)$. Then

\begin{equation}\label{pn*}
2\sum_{i=1}^{n} \mu_s(u_{i})\ge 2\left(q+\mu_{\min}+\mu_{\max}\right)
-\mu_s(u_n)-2\mu_s(u_{2})+\mu_s(u_1). 
\end{equation}

From inequalities (\ref{pn}) and (\ref{pn*}), it follows that
$$
\begin{array}{rcl}
\displaystyle
4\sum_{i=1}^{n} \mu_s(u_{i}) & \ge & 
% 4\left(q+\mu_{\min}+\mu_{\max}\right)
% -\mu_s(u_1)-2\mu_s(u_{n-1})+\mu_s(u_n)-\mu_s(u_n)-2\mu_s(u_{2})+\mu_s(u_1)\\[2.ex]
% &= & 
\displaystyle
4\left(q+\mu_{\min}+\mu_{\max}\right)-2\mu_s(u_{n-1})-2\mu_s(u_{2})\\[2.ex] 
\end{array}
$$

and therefore, 
$$ \displaystyle
\sum_{i=1}^{n} \mu_s(u_{i})\ge q+\mu_{\min}+\mu_{\max}
-\frac{\mu_s(u_{n-1})+\mu_s(u_{2})}{2}\ge q+\mu_{\min},
$$ 
because $\displaystyle \frac{\mu_s(u_{n-1})+\mu_s(u_{2})}{2}\le \mu_{\max}$. 
Then, by applying  Lemma~\ref{cotamuspath}, item $(i)$ follows.

Furthermore, since $\displaystyle q+\mu_{\min}=\sum_{i=1}^{n} \mu_s(u_{i})$ 
then inequality (\ref{pn}) leads us to 
$$
2q+2\mu_{\min}=2\sum_{i=1}^{n} \mu_s(u_{i})\ge 2\left(q+\mu_{\min}+\mu_{\max}\right)
-\mu_s(u_1)-2\mu_s(u_{n-1})+\mu_s(u_n),
$$
following $2\mu_{\max}\le \mu_s(u_1)+2\mu_s(u_{n-1})-\mu_s(u_n).$

Analogously, from inequality (\ref{pn*}), we obtain $2\mu_{\max}\le \mu_s(u_n)+
2\mu_s(u_{2})-\mu_s(u_1).$ Hence, $4\mu_{\max}\le 2\mu_s(u_{2})+ 2\mu_s(u_{n-1})$ 
and therefore, $\mu_{\max}=\mu_s(u_{2})=\mu_s(u_{n-1}).$ Finally, $\mu_s(u_1)\ge 
\mu_s(u_2)=\mu_{\max}$ and $\mu_s(u_{n})\ge \mu_s(u_{n-1})=\mu_{\max}$. Concluding 
that item $(ii)$ holds. 

\end{proof}

\section{Practical application: ensuring robustness through strong-neighbors Roman domination}

In the context of modern sensor networks, ensuring reliable communication and data transmission 
is crucial, especially in environments where sensor failures can lead to significant operational 
disruptions. A sensor network typically consists of numerous sensors, each represented as a 
vertex in a graph. The edges between vertices denote the communication links between sensors, and 
the strength of these links can vary depending on factors such as distance, signal interference, 
and environmental conditions.

In this model, each vertex represents a sensor node in the network. These sensors are responsible 
for monitoring environmental conditions, transmitting data, or performing specific tasks based on 
the network’s purpose. The edges between the vertices represent communication links or paths through 
which data is transmitted between sensors. These edges can have a fuzzy value assigned representing 
the normalized strength or reliability of the communication link. A higher weight implies a more reliable or 
stronger connection, while a lower weight indicates a weaker or less reliable connection.

An edge in this context is considered "strong" if it represents the most reliable communication 
link between two nodes, compared to alternative paths that might involve intermediary nodes. This 
reliability could be due to factors such as close proximity between sensors, strong signal strength, 
or low interference. The importance of these strong edges lies in their ability to ensure that critical 
data is transmitted directly and efficiently, minimizing the risk of loss or corruption associated 
with less reliable paths.

Nodes in the network may have different capabilities. Some nodes can only gather information and 
transmit it to a central data collection point outside the monitoring network. These nodes are labeled 
with a $1.$ Other nodes can both collect information themselves and receive data from other sensors, 
then relay it to the central point. These nodes are labeled with a $2.$ Finally, there are sensors 
labeled with a $0,$ which can only gather information and transmit it to their stronger neighbors—those 
labeled with a $2.$

In the framework of strong-neighbors Roman domination, a strong edge plays a vital role in protecting 
the network. For instance, if a sensor node is connected to its neighbors through strong edges, 
it indicates that even if some sensors fail, the remaining strong connections can still maintain the 
integrity of the network. This reduces the likelihood of network fragmentation and ensures continuous 
operation.

The strong-neighbors Roman domination number, denoted by $\gamma_{snR}(G)$, is a measure used to determine 
how well a sensor network can be protected under single failures. In this situation, the concept can 
be applied to ensure that each sensor node is adequately assisted by the remaining vertices, especially when 
some connections might be unreliable. 

By ensuring that every sensor with a zero label has neighboring sensors capable of providing the 
necessary help, the network can maintain its functionality even in the event of individual sensor 
failures.

Finding the optimal strong neighbor Roman domination number involves determining the minimum amount of 
resources (in terms of reliability) needed to ensure that the network remains operational under adverse 
conditions. An optimal solution would minimize the sum of the labels assigned to the vertices 
(i.e., the overall ``cost'' of protection) while still satisfying the domination conditions. This directly 
translates to using the least amount of energy, backup power, or redundancy necessary to achieve a 
robust and reliable sensor network.

The practical implications of achieving an optimal solution are significant:
\begin{enumerate}
\item Energy efficiency: In sensor networks, energy consumption is a critical concern, especially for 
battery-powered sensors. By optimizing the defense configuration using the strong-neighbors Roman 
domination model, the network can conserve energy, extending the operational lifetime of the sensors.

\item Fault tolerance: The model ensures that the network is resilient to node failures. Even if a critical 
sensor fails, the remaining nodes can compensate for the loss, preventing the network from becoming 
partitioned or experiencing significant performance degradation.

\item Cost-effectiveness: By minimizing the resources allocated to network defense, the overall cost of 
maintaining the network is reduced. This is particularly important in large-scale deployments where 
resources such as backup sensors, power supplies, or repair teams are limited.
\end{enumerate}

As a matter of example, we may describe the following scenario. 
Consider a sensor network deployed in a smart city environment to monitor traffic conditions. Each 
sensor node gathers data on traffic flow, vehicle speed, and congestion levels. In this context, some 
sensors are more critical than others, such as those located at major intersections. These critical sensors 
need to remain operational at all times to ensure the smooth functioning of the traffic management system.

Using the strong-neighbors Roman domination model, the network is analyzed to determine the minimum 
level of redundancy required to keep these critical sensors operational. Sensors at key intersections might 
be assigned a zero label, indicating their critical nature and the need for strong nearby sensors, labelled 
with a two, that can take over if the primary sensor fails. Less critical sensors, such as those monitoring 
quieter streets, might receive a $1$ label, reducing the overall energy consumption and cost of the network.

In this way, the strong-neighbors Roman domination model provides a robust framework for enhancing the 
reliability and efficiency of sensor networks, ensuring that they can operate effectively even in the face 
of potential failures.

\section{Conclusions and future research directions}

In this paper, we extended the concept of Roman domination to fuzzy graphs by introducing and analyzing 
the strong-neighbors Roman domination number in this new context. Throughout our study, we demonstrated 
several important properties of this novel parameter, as well as its relationship with other domination
numbers in fuzzy graphs.

Among the key results obtained, we provide upper and lower bounds for the strong-neighbors Roman 
domination number in specific fuzzy graphs, including complete and complete bipartite graphs. Furthermore, 
we characterize those fuzzy graphs that achieve extreme values of this parameter, significantly 
contributing to the understanding of fuzzy graph structures under this perspective.

Our approach offers a new tool for addressing practical problems in networks where not all connections 
are equally important. This is particularly relevant in resource optimization in communication systems 
and other applications where Graph Theory is crucial. 

Future research could focus on generalizing these concepts to other, more complex fuzzy graph structures 
or on exploring efficient algorithms to compute the strong-neighbors Roman domination number in large-scale 
fuzzy graphs. 

We outline some promising research directions for further exploration.

\begin{itemize}
    \item Develop heuristics or algorithms that enable us to obtain fairly accurate values
    of the strong-neighbors Roman domination number for certain families of graphs.

    \item There is evidence demonstrating that bounding the value of the Roman domination 
    number, known as the RDF decision problem, remains NP-complete. This complexity holds 
    even under the constraints for the underlying graph to be, for example, chordal, 
    bipartite, split or planar. The strong-neighbors Roman domination decision problem 
    will be NP-complete whenever the original Roman domination problem is. However, what 
    would be in general the algorithmic, complexity and approximation properties of 
    the strong-neighbors Roman domination decision problem?
    
    \item Investigating the applicability of this model to real-world networks, such as 
social networks or biological systems, could yield valuable insights.

\end{itemize}

\section*{Author contributions statement} 
Conceptualization, J. C. Valenzuela-Tripodoro, Pedro Garcia-Vazquez and Martin Cera ; Investigation, 
J. C. Valenzuela-Tripodoro, Pedro Garcia-Vazquez and Martin Cera ; Methodology, 
J. C. Valenzuela-Tripodoro, Pedro Garcia-Vazquez and Martin Cera ; Writing – original draft, 
J. C. Valenzuela-Tripodoro, Pedro Garcia-Vazquez and Martin Cera ; Writing – review \& editing, 
J. C. Valenzuela-Tripodoro, Pedro Garcia-Vazquez and Martin Cera.

\section*{Conflicts of interest} 
The authors declare no conflict of interest.

\section*{Data availability} 
No data was used in this investigation.

\end{document}